\documentclass[a4paper,10pt]{article}
\usepackage{geometry}                
\usepackage[active]{srcltx}
\usepackage{hyperref}
\usepackage[normalem]{ulem}

\usepackage{amsmath}
\usepackage{amssymb}
\usepackage{parskip}
\usepackage{mathtools}
\usepackage{fullpage}

\usepackage{mathrsfs}
\usepackage{amsthm}
\usepackage[title]{appendix}
\usepackage{physics}
\usepackage{bbm}

\usepackage[utf8]{inputenc}
\usepackage{caption}
\usepackage{subcaption}
\usepackage{graphicx}
\usepackage{attachfile}
\usepackage{navigator}
\usepackage{theoremref}
\usepackage{faktor}
\usepackage{musicography}
\usepackage{accents}
\usepackage{tikz}
\usetikzlibrary{angles,quotes}

\newtheorem{theorem}{Theorem}[section]
\newtheorem{maintheorem}{Theorem}
\newtheorem{question}{Question}

\newtheorem{definition}[theorem]{Definition}
\newtheorem{proposition}[theorem]{Proposition}

\newtheorem{lemma}[theorem]{Lemma}
\newtheorem{remark}[theorem]{Remark}

\newtheorem{assumption}[theorem]{Assumption}

\numberwithin{equation}{section}

\newcommand{\R}{\mathbb{R}}
\newcommand{\Q}{\mathbb{Q}}

\newcommand{\N}{\mathbb{N}}
\newcommand{\Z}{\mathbb{Z}}

\newcommand{\e}{\varepsilon}

\newcommand{\supp}{\mathrm{supp}}
\newcommand{\Lip}{\mathrm{Lip}}
\newcommand{\diam}{\mathrm{diam}}
\newcommand{\dist}{\mathrm{dist}}

\newcommand{\vol}{\mathrm{vol}}

\newcommand{\Ch}{\mathrm{Ch}}
\newcommand{\cd}{\mathsf{CD}}
\newcommand{\rcd}{\mathsf{RCD}}

\newcommand{\sfd}{\mathsf d}
\newcommand{\di}{{\rm d}}

\newcommand{\mm}{\mathfrak{m}}

\title{The infinitesimal structure of manifolds with non-continuous Riemannian metrics}
\author{Vanessa Ryborz\thanks{University of Oxford. email: vanessa.ryborz@chch.ox.ac.uk \\
ORCID: 0009-0009-4578-3620
}}
\date{}

\begin{document}

\maketitle
\begin{abstract}
 This paper investigates the failure of certain metric measure spaces to be infinitesimally Hilbertian or quasi-Riemannian manifolds, by constructing examples arising from a manifold \( M \) endowed with a Riemannian metric  $g$ that is possibly discontinuous, with $ g, g^{-1} \in L^\infty_{\mathrm{loc}} $ and $ g \in W^{1,p}_{\mathrm{loc}} $ for $ p \leq \dim M - 1 $.
 \let\thefootnote\relax
\end{abstract}

\tableofcontents

{\small \textbf{Acknowledgements.} 
The author gratefully acknowledges Andrea Mondino for his supervision and guidance, and thanks Alessandro Cucinotta and Alexander Lytchak for valuable discussions. The author is supported by a postgraduate scholarship from the Mathematical Institute at the University of Oxford. Further financial support from the Deutsche Forschungsgemeinschaft (DFG, German Research Foundation) under Germany’s Excellence Strategy – EXC-2047/1 – 390685813, which funded the 2025 Trimester Program ``Metric Analysis” in Bonn, is also acknowledged with gratitude.
}

\section{Introduction}
Manifolds with (semi)-Riemannian metrics of low regularity (in particular below $C^2$) have been of growing interest and importance in the field of geometric analysis. They may arise from gluing constructions, or in the case of Lorentzian signature $(-, +, \ldots, +)$, serve to understand physically relevant phenomena from general relativity (see for instance \cite{calisti2025hawking}). 

One approach to generalise Riemannian manifolds is to drop any underlying smoothness assumption and study metric measure spaces. A metric measure space is a triplet $(X, \sfd, \mathfrak{m})$, where $(X, \sfd)$ is a complete and separable metric space and $\mathfrak{m}$ is a $\sigma$-finite Borel measure on $X$ (playing the role of the volume measure). For a similar generalisation of Lorentzian manifolds (see \cite{kunzinger2018lorentzian, cavalletti2020optimal}).

Various concepts from Riemannian geometry can be generalised to metric measure spaces such as synthetic notions of the Ricci curvature bounded below by $K$ and the dimension bounded above by $N$ for some $K \in \R$ and $N \in [1, \infty]$; the curvature dimension condition $\cd(K, N)$ defined by Lott-Villani \cite{lott2009ricci, LottVillaniJFA} and Sturm \cite{sturm2006geometryI, sturm2006geometryII}.    

In this paper however, we will focus on the generalised first order calculus on metric measure spaces.
First approaches were made by Cheeger \cite{cheeger1999differentiability} to generalise Rademacher's theorem to PI-spaces, Haj\l{}asz \cite{hajlasz1996sobolev} to prove generalised Sobolev embedding theorems, and Shanmugalingam \cite{shanmugalingam2000newtonian}. 
Later, Ambrosio, Gigli and Savaré introduced a Sobolev calculus based on the notion of test plans in \cite{ambrosio2014inventio} that they proved to be equivalent to the approaches in \cite{cheeger1999differentiability} and \cite{shanmugalingam2000newtonian}, leading to the definition of the Sobolev space $W^{1,2}(X)$, and a Laplace operator. 

In \cite{gigli2015differential}, Gigli studied conditions under which the Laplacian is linear and proposed the notion of \textit{infinitesimal Hilbertianity}. A metric measure space is infinitesimally Hilbertian if the associated Sobolev space $W^{1,2}(X)$ is a Hilbert space, or equivalently if the Cheeger energy (a generalisation of the Dirichlet energy to metric measure spaces) is a quadratic form. Then the Laplacian is automatically linear and the theory on Dirichlet forms provides various tools to better understand the Laplacian and the heat flow.   
 
Examples of infinitesimally Hilbertian metric measure spaces include geodesically complete weighted Riemannian manifolds \cite{luvcic2020infinitesimal}, locally $\mathrm{CAT}(\kappa)$-spaces equipped with any Borel measure  \cite{di2021infinitesimal}, and sub-Riemannian manifolds equipped with any Borel measure \cite{le2023universal}.
Relevant notions to study the Cheeger energy are the metric speed of absolutely continuous curves (\cite{ambrosio1990metric}, \cite{kirchheim1994rectifiable}) and the slope of Lipschitz continuous functions. 

In the context of curvature dimension conditions the first order calculus on metric measure spaces turned out to be a crucial ingredient in the search for a criterion that distinguishes Riemannian-like structures from Finslerian structures.
In \cite{ambrosio2014duke}, Ambrosio, Gigli, and Savaré proposed the Riemannian curvature dimension condition, $\rcd(K, \infty)$, a refinement of the curvature dimension condition which combines the $\cd(K, \infty)$-condition with infinitesimal Hilbertianity.
The $\rcd(K, N)$-condition for $N<\infty$ was proposed by Gigli in \cite{gigli2014overview}. 
Using this terminology, the metric measure space $(\R^n, \norm{\cdot}, \mathcal{L}^n)$ is an $\rcd(0, n)$-space if and only if $\norm{\cdot}$ arises from an inner product, hence in that case infinitesimal Hilbertianity indeed rules out Riemannian structures among Finslerian ones.

A smooth $d$-manifold $M$, equipped with a continuous Riemannian metric $g$ can naturally be seen as a metric measure space, where the measure is induced by the volume form (locally given via $\di\vol_g= \sqrt{\det g}\, \di \mathcal{L}^d$) and the Riemannian distance $\sfd_g$ is defined as the infimum over lengths of piecewise smooth curves (see \eqref{def:dg}).

However, there are several relevant classes of Riemannian metrics that are non-continuous, for instance the Geroch-Traschen class \cite{geroch1987strings}, i.e., those Riemannian metrics $g$ such that $g, g^{-1} \in L^\infty_{\rm loc}(M)$ and $g \in W^{1,2}_{\rm loc}(M)$. These regularity assumptions are minimal to compute the distributional Ricci curvature on $M$ (see \cite{grosser2013geometric}, \cite{graf2020singularity}) and have been studied in both Riemannian (see for example \cite{grant2011synthetic}) and Lorentzian signature (see \cite{steinbauer2009geroch}). Moreover, in the context of potential theory for uniformly elliptic operators, the even broader class of (Riemannian) metrics $g$ such that $g, g^{-1} \in L^\infty_{\rm loc}(M)$ has been studied in \cite{saloff1992uniformly}, \cite{sturm1995analysis}, \cite{norris1997heat} and \cite{de1991integral}. 

If the metric tensor is only measurable, more precisely an equivalence class of metric tensors that is defined up to local $\mathcal{L}^d$-almost everywhere equality, Norris (cf. \cite{norris1997heat}) and De Cecco-Palmieri (cf. \cite{de1991integral}) defined a distance $\sfd_g$ that avoids the difficulty that piecewise smooth curves $\gamma$ might lie on a null set where $g$ is not well-defined (see Section \ref{sec:3}). 

This paper is driven by the following questions, which arise naturally from the results on smooth Riemannian manifolds: 
\begin{question}[see also \cite{de1990conversazioni}]\label{q1}
    Under which regularity assumptions on $g$ does $(M, \sfd_g, g)$ satisfy that for any Lipschitz continuous function $f:M \to \R$ its slope equals the norm of its gradient almost everywhere, i.e., $|Df|(x)=\limsup_{y \to x} \frac{|f(x)-f(y)|}{\sfd_g(x,y)}= |\nabla_g f|_g(x)$ almost everywhere?
\end{question}
This question has been posed by De Giorgi in \cite{de1990conversazioni} as the positive answer would be a necessary condition for  $(M, \sfd_g, g)$ to be a quasi-Riemannian manifold, as defined in \cite{de1990conversazioni}, see also Definition \ref{def:quasi_riemannian_manifold}.  A related question:
\begin{question}\label{q2}
   Under which regularity assumptions on $g$ does the metric speed of ``almost every" absolutely continuous curve equal the norm of its derivative, i.e.,  $\lim_{h \to 0} \frac{\sfd_g(\gamma_t, \gamma_{t+h})}{|h|}=|\dot{\gamma}_t|_g$ for almost every $t \in [0,1]$?
\end{question}
What exactly ``almost every" curve refers to is clarified in Section \ref{sec:2}. 
\begin{question}\label{q3}
   Under which regularity assumptions on $g$ is $(M, \sfd_g, \vol_g)$ infinitesimally Hilbertian?
\end{question}
This particular question refers back to the original purpose of infinitesimal Hilbertianity to identify the Riemannian-like structures among metric measure spaces. 

It is already known that assuming $g \in C^0$ is sufficient for the second and third question to have positive answers (see \cite{burtscher2012length}, \cite{mondino2025equivalence}). Moreover, it is known that for $g \in C^0$ and $f \in C^1$, the slope $|Df|$ of $f$ equals the norm of its gradient \cite{mondino2025equivalence}.
In \cite{sturm1997diffusion}, Sturm establishes that there exist cases, where the answer to Questions \ref{q1} and \ref{q2} is negative if one does not assume the metric $g$ to be continuous. 
Moreover, in \cite{de1995lip}, De Cecco and Palmieri present one example of a metric $g$ on $\R^2$ with $g, g^{-1} \in L^\infty$ ans such that on a positive measure set of points, the tangent is a normed space, with a norm that does not arise from an inner product. 

Assuming $W^{1,p}_{\rm loc}$-Sobolev regularity with growing exponent $p$ can be seen as a bridge from $L^\infty_{\rm loc}$-metrics to $C^0$-metrics.
Indeed, for $p<p'$, it holds $W^{1,p'}_{\rm loc} \subset W^{1,p}_{\rm loc}$ and by the Sobolev embedding theorem, a Riemannian metric $g$ on a manifold $M$ with $g \in W^{1,p}_{\rm loc}$, if $p > \dim M$, is automatically continuous. This suggests that additional Sobolev regularity could be a discriminant for Questions \ref{q1} to \ref{q3}.

In this paper, we discuss a related example to \cite{de1995lip} that facilitates the study of infinitesimal Hilbertianity, and that allows a refinement to demonstrate that the above questions still may have negative answers if we additionally assume Sobolev regularity of the metric tensor. 

\textbf{Outline of the paper and main results.}
In Section \ref{sec:2}, we will recall some basic notions on the first order calculus on metric measure spaces and their explicit descriptions in the case of a manifold equipped with a continuous Riemannian metric. In Section \ref{sec:3}, we recall the definition of the distance in the case of a non-continuous Riemannian metric by Norris and De Cecco-Palmieri, and briefly study absolutely continuous curves and test plans in that setting. 
In Section \ref{sec:4}, we will construct Riemannian metrics $g$ with $g, g^{-1} \in L^\infty_{\rm loc}$ such that $(M, \sfd_g, g)$ is not a quasi-Riemannian manifold and such that $(M, \sfd_g, \vol_g)$ is not infinitesimally Hilbertian. We start with metrics that only satisfy $g, g^{-1} \in L^\infty_{\rm loc}$, providing an introductory example of a non-quasi Riemannian manifold in Subsection \ref{subsec:first_counter_quasi_riem} and an introductory example of a non-infinitesimally Hilbertian space in Subsection \ref{subsec:first_counter_inf_hilbertian}. These examples are related to the one in \cite{de1995lip}.
In Subsection \ref{subsec:sobolev}, we refine the previous examples to a Sobolev metric $g$ such that $g, g^{-1} \in W^{1,p}_{\rm loc}\cap L^\infty_{\rm loc}$ for some $p \leq \dim M -1$. 
Questions \ref{q1} and \ref{q2} are then addressed in Subsection \ref{subsub_quasi_riem}, where we prove:
\smallskip

\begin{maintheorem}\label{snd_counterexample}
    For $d \geq 3$, $p \in [1, d-1]$ there exists a $d$-dimensional manifold $M$ and a Riemannian metric $g$ on $M$ such that $g, g^{-1} \in L^\infty$, $g \in W^{1,p}_{\rm loc}(M)$ and such that 
    \begin{itemize}
        \item[(i)] There exists a test plan that gives positive measure to curves for which the metric speed is strictly smaller than the norm of the derivative for almost every time $t \in [0,1]$. 
        \item[(ii)] There exists a function $f \in C^1(M)$ for which the minimal weak upper gradient and the metric slope are strictly larger than the norm of the gradient on a set of positive measure, i.e. $|Df|, |Df|_w > |\nabla_g f|_g$ on a set of positive $\vol_g$-measure.
    \end{itemize}
\end{maintheorem}
Question \ref{q3} is addressed in Subsection \ref{subsub:hilbertian}, where we prove:
\smallskip

\begin{maintheorem}\label{hilbertianity_counterexample}
    For $d\geq 3$, $p\in [1, d-1]$ there exists a $d$-dimensional manifold $(M, g)$ with $g, g^{-1} \in L^\infty_{\rm loc}$ and $g \in W^{1,p}_{\rm loc}(M)$ such that the metric measure space $(M, \sfd_g, \vol_g)$ is not infinitesimally Hilbertian.
\end{maintheorem}
\subsection*{Notation} 
Given a manifold $M$, a Riemannian metric $g$, and a locally Lipschitz continuous function $f:M \to \R$, $|Df|$ will always denote its slope, $|Df|_w$ its minimal weak upper gradient, $\nabla_gf$ its gradient with respect to $g$ and $\nabla_{euc} f$ its differential with respect to chosen local coordinates (which will be specified). \\
Given an absolutely continuous curve $\gamma:[0,1] \to M$, we will for almost every $t \in [0,1]$ denote its metric speed at time $t$ by $|\dot{\gamma}_t|$, and its derivative at time $t$ by $\dot{\gamma}_t \in T_{\gamma_t}M$. To avoid confusion, if we take the norm of the derivative, we will always write $|\dot{\gamma}_t|_g$ for $\sqrt{g(\dot{\gamma}_t,\dot{\gamma}_t)}$ and $|\dot{\gamma}_t|_{euc}$ for $\sqrt{\langle\dot{\gamma}_t,\dot{\gamma}_t\rangle_{euc}}$ (given specified local coordinates). \\
When working in $\R^d$, $e_i$ will denote the $i$-th unit vector, $\mathcal{L}^d$ the $d$-dim Lebesgue measure, and $\mathrm{Id}$ will denote the identity matrix. For $i \in \{1, \ldots, d\}$ we will denote by $x_i: \R^d \to \R$ the projection on the $i$-th coordinate. \\
For any measure space $(X, \mathfrak{m})$ and an $\mathfrak{m}$-measurable set $A \subset X$, $\mathbbm{1}_{A}$ will denote the characteristic function of $A$. 
\section{Basic definitions on metric measure spaces}\label{sec:2}
In this section we will recall notions from first order calculus on metric measure spaces and their identification in metric measure spaces that arise from manifolds equipped with an at least continuous Riemannian metric. 
\subsection{First order calculus on metric measure spaces}\label{subsec:2.1}
We will briefly summarise some tools to define a generalised notion of modulus of gradients and Sobolev functions in metric measure spaces that was introduced by Cheeger in \cite{cheeger1999differentiability} and further analysed by Ambrosio, Gigli and Savaré  \cite{ambrosio2014inventio}. The structure of this section has been inspired by \cite{mondino2025equivalence}.

Let $(X, \sfd)$ be a complete and separable metric space. The slope (or local Lipschitz constant) of a real valued function $f: X \to \R$ is defined by
\begin{align*}
    |Df|(x) := \limsup_{y \to x}\frac{|f(x)-f(y)|}{\sfd(x, y)},
\end{align*}
if $\{x\}$ is not isolated, and $0$ otherwise.

We endow $(X,\sfd)$ with a non-negative $\sigma$-finite Borel  measure $\mathfrak{m}$, obtaining the metric measure space $(X, \sfd, \mathfrak{m})$.
Throughout the rest of this work, we assume that there exists a bounded Borel Lipschitz map $V: X \to [0, \infty)$ such that  (cf.\@ \cite[Section\,4]{ambrosio2014inventio})
\begin{equation}\label{condition_on_measure_bounded_compression}
\begin{split}
    V\ \mathrm{is\ bounded\ on\ each\ compact\ set\ } K \subset X \ \mathrm{and}  \\
    \int_X e^{-V^2} \, \di\mathfrak{m} \leq 1. 
    \end{split}
\end{equation}
\begin{definition}
    Let $\gamma \in C([0,1],X)$ be a curve. Then we say that $\gamma$ is absolutely continuous if there exists a function $g \in L^1([0,1], [0, \infty))$ such that for all $0 \leq s<t \leq 1$, it holds
    \begin{align*}
        \sfd(\gamma_s, \gamma_t) \leq \int_s^t g(r) \, \di r.
    \end{align*}
    There exists a minimal such function $g$, which we call the metric speed of $\gamma$ and denote it by $|\dot{\gamma}_t|$. 
\end{definition}
The metric speed of an absolutely continuous curve $\gamma:[0,1] \to X$ is given by  $|\dot{\gamma}_t|:= \lim_{h \to 0}\frac{\sfd(\gamma_{t+h}, \gamma_t)}{|h|}$ {which exists for almost every $t \in [0,1]$} \cite{ambrosio1990metric}, (see also \cite{kirchheim1994rectifiable}).

We next recall the notion of test plan and weak upper gradient. We will use the conventions of  \cite{gigli2018nonsmooth}, after \cite{ambrosio2014inventio}. For any metric space $(Y, \sfd_Y)$, we denote by $\mathcal{P}(Y)$ the set of Borel probability measures on $Y$. Given two topological spaces $X_1, X_2$, a Borel-measurable map $T:X_1 \to X_2$, and a Borel measure $\mu$ on $X_1$, the pushforward measure $T_\#\mu$ on $X_2$ is the Borel measure defined by $T_\#\mu(B) = \mu (T^{-1}(B))$ for any Borel set $B \subset X_2$.

For $t \in [0,1]$ the evaluation map $e_t: C([0,1], X) \to X$ is defined via $e_t(\gamma) = \gamma_t$. 
\smallskip

\begin{definition}\label{test_plan}
    Let $\boldsymbol{\pi} \in \mathcal{P}(C([0, 1], X))$. We say that $\boldsymbol{\pi}$ is a \emph{test plan} if there exists a constant $C(\boldsymbol{\pi})>0$ such that 
    \begin{align*}
        (e_t)_\#\boldsymbol{\pi} \leq C(\boldsymbol{\pi})\mathfrak{m}\ \mathrm{for \ all \ } t \in [0,1]
    \end{align*}
    and 
    \begin{align*}
        \int \int_0^1 |\dot{\gamma}_t|^2\,\di t\di\boldsymbol{\pi}(\gamma) < \infty. 
    \end{align*}
    We use the convention that if $\gamma$ is not absolutely continuous, then $\int_0^1 |\dot{\gamma}_t|^2\,\di t =\infty$.
\end{definition} 
In Question \ref{q2}, ``almost every absolutely continuous curve" refers to a set of curves $\Gamma \subset C([0,1], M)$ such that for every test plan $\boldsymbol{\pi}$ as defined above, $\boldsymbol{\pi}(\Gamma)=1$.

In the following two definitions, we say $\mm$-measurable functions to refer to equivalence classes of $\mm$-measurable funcrtions up to $\mm$-a.e. equality. 
\smallskip

\begin{definition}
   Given { $f: X \to \R$ } a $\mathfrak{m}$-measurable function, then a $\mathfrak{m}$-measurable function $G: X \to [0, \infty]$ is called a weak upper gradient of $f$, if 
   \begin{align}\label{weakgradient}
       \int|f(\gamma_1)-f(\gamma_0)|\,\di\boldsymbol{\pi}(\gamma)\,  \leq \int\int_0^1G(\gamma_t)|\dot{\gamma}_t|\, \di t\di\boldsymbol{\pi}(\gamma) < \infty, \quad \mathrm{for \ all \ test \ plans } \ \boldsymbol{\pi}. 
   \end{align}
\end{definition}
The discussion in \cite[Proposition 5.9 and Definition 5.11]{ambrosio2014inventio}, shows the existence of a weak upper gradient $|Df|_w$ such that $|Df|_w \leq G$ {$\mathfrak{m}$-a.e.} for all other weak upper gradients $G$. We will call it the \textit{minimal weak upper gradient} of $f$.
\begin{definition}
    The Cheeger energy is defined in the class of $\mathfrak{m}$-measurable functions by
    \begin{align*}
        \Ch(f):= \left\{ \begin{array}{ll}
           \frac{1}{2}\int_X |Df|^2_w\, \di\mathfrak{m}  & \mathrm{if}\ f\ \mathrm{has\ a\ weak\ upper\ gradient\ in }\ L^2(X, \mathfrak{m}),  \\
           \infty  & \mathrm{otherwise},
        \end{array}\right.
    \end{align*}
    with proper domain $D(\Ch)= \{f:X \to [0, \infty], \mathfrak{m}\mathrm{-measurable}, \ \Ch(f) < \infty\}$.
\end{definition}
\begin{definition}
    A metric measure space is called \emph{infinitesimally Hilbertian} if  $\Ch$ is a quadratic form.
\end{definition}
\subsection{Manifolds with continuous Riemannian metrics}\label{subsec:2.2}
Consider a $d$-dimensional manifold $M$ equipped with a Riemannian metric $g \in C^{0}(M)$. Then $g$ induces a volume measure whose density is locally given by $\di\vol_g = \sqrt{\det g}\, \di \mathcal{L}^d$. 
The distance induced by $g$ is given by 
\begin{align}\label{def:dg}
    \sfd_g(x, y) = \inf\Big\{ \int_{0}^1 |\dot{\gamma_t}|_g\, \di t: \, \gamma\ \mathrm{piecewise \ }C^\infty, \gamma_0=x, \gamma_1 = y \Big\}.
\end{align}
In \cite{burtscher2012length} the metric structure for such spaces $(M, \sfd_g)$ has been thoroughly studied. In \cite{mondino2025equivalence}, the Cheeger energy and related notions were studied and expressed in terms of classical Sobolev spaces on manifolds (cf. \cite{hebey1996sobolev}). We will recall some of these results to point out that in the case of a Riemannian manifold with a continuous metric, several notions from Subsection \ref{subsec:2.1} can be identified with expected classical objects from calculus on manifolds. 
\smallskip

\begin{proposition}[\cite{burtscher2012length} Proposition.\;4.1 and Theorem\;3.15]\label{length_space}
    The metric \eqref{def:dg} turns $M$ into a length space.
\end{proposition} 
 A standard fact on the slope of functions is that if $f \in C^1(M)$, then  $|Df|(x) = |\nabla_g f|_g(x)$ for every $x$.
\smallskip
 
\begin{proposition}[\cite{burtscher2012length}, Proposition 4.10]\label{metric_speed_explicit}
    Let $\gamma:[0, 1]\to M$ be an absolutely continuous curve. Then the metric speed $|\dot{\gamma_t}|$ coincides a.e.\;with $|\dot{\gamma}_t|_g = \sqrt{\langle  \dot{\gamma}_t, \dot{\gamma}_t\rangle_g}$, where $\dot{\gamma}_t$ denotes the (a.e.\;existing) derivative. 
\end{proposition}
In the case of a manifold with a continuous Riemannian metric, it also turns out that the minimal weak upper gradient of a $C^1$-function $f$ can be identified as follows:
\begin{proposition}[\cite{mondino2025equivalence} Proposition 4.24]\label{weakgradientforc1}
    Let $M$ be a smooth manifold, $g$ a $C^0$-Riemannian metric on $M$.
    Let $f$ be a $C^1$-function. Then 
    \begin{align*}
        |Df|_w(x) = |\nabla_g f|_g(x), \quad {\text{for $\vol_g$-a.e. x.}}
    \end{align*}
\end{proposition}
That was a crucial ingredient to conclude that: 
\smallskip

\begin{proposition}[\cite{mondino2025equivalence}, Corollary 4.27]
    Let $M$ be a manifold equipped with a continuous Riemannian metric $g$. Then the metric measure space $(M, \sfd_g, \vol_g)$ is infinitesimally Hilbertian. 
\end{proposition}

\section{The distance on manifolds with non-continuous Riemannian metrics}\label{sec:3}
From now on we will consider a smooth $d$-dimensional manifold $M$ with a Riemannian metric $g$ such that $g, g^{-1} \in L^\infty_{\rm loc}(M)$. 
Let $U \subset M$ be open such that $U$ lies in one coordinate patch and let $\psi: U \to \R^d$ be a local trivialisation. Then, $g, g^{-1} \in L^\infty_{\rm loc}(M)$ means that for every such $U$, the components of $\psi_*g, \psi_*g^{-1}$ are equivalence classes of locally $\mathcal{L}^d$-measurable functions up to $\mathcal{L}^d$-almost everywhere equality in $\psi(U)$ and that for every compact set $K \subset \psi(U)$, there exist constants $\lambda, \Lambda >0$ such that $\lambda \mathrm{Id} \leq \psi_*g \leq \Lambda \mathrm{Id}$ $\mathcal{L}^d$-almost everywhere in $K$. Note that under these assumptions the measures $\mathcal{L}^d$ and $\vol_g$ are locally equivalent and hence have the same null sets.

Then $(M, g)$ falls into the class of LIP-Riemannian manifolds (see for example \cite{teleman1983index, de1991integral}). These manifolds and their metric structures, have been studied by Norris \cite{norris1997heat}, and De Cecco and Palmieri  \cite{de1991integral}, as well as  Saloff-Coste  \cite{saloff1992uniformly} and Sturm  \cite{sturm1995analysis} with emphasis on potential theory for uniformly elliptic operators. 
Note that LIP-manifolds shall not be confused with metric tensors whose coefficients are locally Lipschitz continuous.

If $g$ is only measurable, it might not be well-defined everywhere, hence the definition of the metric $\sfd_g$ from the continuous case \eqref{def:dg}, cannot be used for non-continuous Riemannian metrics. 
However, Norris \cite{norris1997heat} and De Cecco-Palmieri \cite{de1991integral} gave a definition of a distance $\sfd_g$ on $M$ arising from $g$ that avoids this difficulty. We will briefly summarise their results and then study the metric speed of curves, test plans, and the Cheeger energy on the corresponding metric measure space. 

In the case of a Riemannian manifold with $g, g^{-1}\in L^\infty_{\rm loc}$, Norris defines the distance $\sfd$ via
 \begin{align*}
     \sfd(x,y) := \sup \left\{  w(x)-w(y)  ,\ w \in C^{0,1}_{\rm loc}(M)\ \mathrm{and}\ |\nabla_g w|_g\leq 1\ \vol_g\mathrm{-a.e. } \right\}.
 \end{align*}
Here, $w \in C^{0,1}_{\rm loc}(M)$ means that the function $w$ is locally Lipschitz continuous with respect to the Euclidean distance on charts.
 
 Moreover, for $\rho\in C_c^\infty(B_1^{\R^d}(0), [0,\infty))$, with $\int_{\R^d} \rho\, \di x =1$, define $\rho_\e := \frac{1}{\e}\rho(\frac{x}{\e})$ and $g_\e := \rho_\e * g$ (for details on the convolution on manifolds see \cite{grosser2013geometric, graf2020singularity}). The metrics $g_\e$ are continuous and hence the distance $\sfd_{g_\e}$ is given as in \eqref{def:dg}. Norris then proves that for $\e \to 0$, $\sfd_{g_\e}$ converges to a distance $\sfd_0$ that is locally equivalent to the underlying Euclidean metric induced from a coordinate chart; in particular,  \cite[Theorem 3.6]{norris1997heat} states that $\sfd_0 = \sfd =: \sfd_g$. 

 From the properties of $g$, it follows that all points outside a $\vol_g$-null set $N_g \subset M$ are Lebesgue points of $g, g^{-1}$. For any $\vol_g$-null set $N \subset M$, define the set 
 \begin{align}
     \mathrm{Lip}_N(x, y, M):= \{\gamma \in C^{0,1}_{\rm{loc}}([0,1], M): \gamma(0)=x, \gamma(1)=y, \mathcal{L}^1(\{t \in [0,1]: \gamma(t) \in N\}) =0\}.
 \end{align}
 Here $C^{0,1}_{\rm loc}([0,1], M)$ refers to locally Lipschitz continuous curves with respect to Euclidean distance induced by the local charts. 
 If $\gamma \in C([0,1], M)$ satisfies $\mathcal{L}^1(\{t \in [0,1]: \gamma(t) \in N\}) =0$, we say that $\gamma$ is transversal to $N$ and write $\gamma \perp N$.

From now on, we will make a further assumption, that will trivially hold for all the relevant example manifolds in this paper. 

\smallskip
\begin{assumption}[\cite{de1991integral}, (1.13)]\label{assumption_on_mfd}
The manifold $(M,g)$ satisfies the following:
    \begin{itemize}
    \item[(i)] The manifold $M$ has infinite $\vol_g$-volume. 
    \item[(ii)] There exists a constant $h \geq 0$ such that all coordinate charts in our atlas satisfy that the smallest eigenvalue of $g$ is uniformly bounded below $h$ and $M$ is complete or the interior of a complete manifold (where the completeness is with respect to the distance induced by the charts as in \cite[(1.9)]{de1991integral}).
\end{itemize}
\end{assumption}

Assumption \ref{assumption_on_mfd} is relevant to most results in \cite{de1991integral}, in particular in establishing that two different definitions of the distance are equal. 
Indeed, De Cecco and Palmieri proved in \cite[Theorem (2.18) and Theorem (6.1)]{de1991integral}  that under Assumption \ref{assumption_on_mfd}, it holds $\sfd_g= \boldsymbol{\delta}$, where the distance $\boldsymbol{\delta}$ is given by 
 \begin{align}\label{null_set_def_of_metric}
    \boldsymbol{\delta}(x, y)= \sup_{N\subset M, \vol_g(N)=0} \boldsymbol{\delta}_N(x, y),
 \end{align}
where
 \begin{align}
     \boldsymbol{\delta}_N :=\inf\{L_g(\gamma), \gamma \in \Lip_N(x, y, M)\},
 \end{align}
 and $L_g(\gamma) = \int_0^1 |\dot{\gamma}_t|_g \di t$.
 In \cite[Theorem 4.4]{de1988distanza}, it is shown that for each $\vol_g$-null set $N_0 \subset M$, $\boldsymbol{\delta} = \boldsymbol{\delta}_0:= \sup_{N, \vol_g(N)=0} \boldsymbol{\delta}_{N \cup N_0}$. 
\smallskip
  
 \begin{proposition}
     Let $\gamma \in C^{0,1}_{\rm loc}([0, 1] \to M)$ be a curve such that $\gamma \perp N_g$. Then, for almost every $t \in [0,1]$, it holds
     \begin{align*}
          |\dot{\gamma}_t| \leq |\dot{\gamma}_t|_g.
     \end{align*}
 \end{proposition}
 \begin{proof}
     For a Lipschitz curve $\gamma \perp N_g$, and any $t \in [0,1]$, $h \neq 0$, we have that
     \begin{align*}
         \sfd_g(\gamma_t, \gamma_{t+h}) = \lim_{\e \to 0} \sfd_{g_\e}(\gamma_t, \gamma_{t+h}) \leq \lim_{\e \to 0}\int_0^h |\dot{\gamma}_{t+\tau}|_{g_\e}\, \di \tau = \int_0^h |\dot{\gamma}_{t+\tau}|_g\, \di \tau,
     \end{align*}
     where the last equality follows from the dominated convergence theorem and the fact that $\rho_\e*g \to g$ pointwise outside $N_g$. Now the characterisation $|\dot{\gamma}_t|= \lim_{h \to 0} \frac{\sfd_g(\gamma_t, \gamma_{t+h})}{|h|}$ a.e. together with the Lebesgue differentiation theorem yields the proposition. 
 \end{proof}
 There are more observations that can be made on the metric speed of absolutely continuous curves. The first one is that the metric speed depends on the derivative of the curve. 
\smallskip
 
 \begin{proposition}\label{metric_speed_depends_on_direction}
 Let $M$ be a smooth manifold and $g$ a Riemannian metric such that $g, g^{-1} \in L^\infty_{\rm loc}$. 
     Let $\gamma:[0,1] \to M$ be a Lipschitz continuous curve and $t\in [0,1]$ such that the metric speed $|\dot{\gamma}_t|$ exists at $t$ and $\gamma$ is differentiable at $t$. Let $c:[0,1] \to M$ be another curve such that $c_t = \gamma_t$, $c$ is differentiable at $t$ and $\dot{c}_t = \dot{\gamma}_t$. Then the metric speed $|\dot{c}_t|$ of $c$ at $t$ exists and $|\dot{c}_t| = |\dot{\gamma}_t|$. 
 \end{proposition}
 \begin{proof}
     We work in local coordinates in $\R^d$ and write $\gamma_t =  c_t =:p$ and  $\dot{c}_t = \dot{\gamma}_t=v$. Furthermore, we may restrict to a small neighbourhood $U$ around $p$ and assume that there exist $\lambda, \Lambda>0$ such that $\lambda \mathrm{Id} \leq g \leq \Lambda \mathrm{Id}$. Then, by the definition of the derivative, we get that 
     \begin{align*}
         \gamma_{t+h} = p + hv + a_{\gamma}(h), \  c_{t+h} = p + hv + a_{c}(h) \in \R^d,
     \end{align*}
     where
     \begin{align}
         \lim_{h \to 0} \frac{|a_{\gamma}(h)|_{euc}+|a_{c}(h)|_{euc}}{|h|} = 0.
     \end{align}
     Now, we can compute that 
     \begin{align}
         \sfd_g(\gamma_{t+h}, c_{t+h}) \leq \sqrt{\Lambda} (|a_{\gamma}(h)|_{euc}+|a_{c}(h)|_{euc}),
     \end{align}
     and hence 
     \begin{align*}
         0 \leq \lim_{h \to 0} \frac{\sfd_g(\gamma_{t+h}, c_{t+h})}{|h|} \leq \lim_{h \to 0} \sqrt{\Lambda} \frac{|a_{\gamma}(h)|_{euc}+|a_{c}(h)|_{euc}}{|h|} =0.
     \end{align*}
     But then 
     \begin{align*}
        \limsup_{h \to 0}\bigg| \frac{\sfd_g(c_t, c_{t+h})}{|h|} -  \frac{\sfd_g(\gamma_t, \gamma_{t+h})}{|h|}\bigg| \leq  \limsup_{h \to 0}  \frac{ \sfd_g(\gamma_{t+h}, c_{t+h})}{|h|} = 0.
     \end{align*}
 \end{proof}
 Next, we establish that the metric speed depends locally Lipschitz-continuously on the derivative of a curve. 
\smallskip
 
 \begin{proposition}\label{metric_speed_continuous_from_direction}
     Let $\gamma^1, \gamma^2 :[0,1] \to M$ be two Lipschitz continuous curves that are contained in a compact set $K$ inside one coordinate patch $V$ and $t \in [0,1]$ such that $\gamma^1_t = \gamma^2_t$ and such that the metric speeds $|\dot{\gamma}^1_t|$ and $|\dot{\gamma}^2_t|$ as well as their derivatives $\dot{\gamma}^1_t,\dot{\gamma}^2_t\in T_{\gamma^1_t}M$ exist. Then there exists a constant $C =C(K)$ such that 
     \begin{align}
         ||\dot{\gamma}^1_t|-|\dot{\gamma}^2_t|| \leq C|\dot{\gamma}^1_t-\dot{\gamma}^2_t|_g.
     \end{align}
 \end{proposition}
 \begin{proof}
     We may assume to be working in $\R^d$. Now there exist $\lambda, \Lambda >0$ such that $\lambda \mathrm{Id} \leq g \leq \Lambda \mathrm{Id}$ on $K$. 
     Denote $v := \dot{\gamma}^1_t$ and $w := \dot{\gamma}^2_t$. By translation we may assume that $t=0$ and denote $p:=\gamma^1_0 = \gamma^2_0$. 
     By the previous proposition, we may assume that $\gamma^1_t = p+tv$ and $\gamma^2_t = p+tw$. But then 
     \begin{align*}
         \Big|\frac{\sfd_g(\gamma^1_0, \gamma^1_h)}{|h|} -\frac{\sfd_g(\gamma^2_0, \gamma^2_h)}{|h|}\Big| &\leq \frac{\sfd_g(p+hv, p+hw)}{|h|} \leq \sqrt{\Lambda} \frac{\sfd_{euc}(p+hv,p+hw)}{|h|} \\
         &= \sqrt{\Lambda}|v-w|_{euc} \leq \sqrt{\frac{\Lambda}{\lambda}}|v-w|_g = \sqrt{\frac{\Lambda}{\lambda}} |\dot{\gamma}^1_t-\dot{\gamma}^2_t|_g.
     \end{align*}
 \end{proof}

 \begin{lemma}\label{testplans_on_transversal}
     Let $\boldsymbol{\pi}\in \mathcal{P}(C([0,1], M))$ be a test plan and $N\subset M$ a $\vol_g$-null set. Then, it holds that 
     \begin{align}
         \boldsymbol{\pi}\Big(\{\gamma \in C([0,1], M): \gamma \perp N \}\Big)=1.
     \end{align}
 \end{lemma}
 \begin{proof}
     We argue by contradiction. Assume that $\boldsymbol{\pi}\Big(\{\gamma \in C([0,1], M): \gamma \perp N \}\Big)<1$. 
     Then as 
     \begin{align*}
         \{\gamma \in C([0,1], M): \gamma \not \perp N \}= \bigcup_{n \in \N} \{\gamma \in C([0,1], M): \mathcal{L}^1(\gamma^{-1}(N)) \geq \frac{1}{n}\},
     \end{align*}
     there exists an $n \in \N$ and an $\e>0$ such that
     \begin{align}
       \boldsymbol{\pi} (\{\gamma \in C([0,1], M): \mathcal{L}^1(\gamma^{-1}(N)) \geq \frac{1}{n}\}) =\e >0.
     \end{align}
     Denote $\{\gamma \in C([0,1], M): \mathcal{L}^1(\gamma^{-1}(N)) \geq \frac{1}{n}\}=:B_n$ and note that for each curve $\gamma \in B_n$, it holds
     \begin{align}\label{null_set_count_along_bad_curve}
         \int_0^1 \mathbbm{1}_{N}(\gamma_t)\,\di t \geq \frac{1}{n},
     \end{align}
     hence, with Fubini's theorem, we get that 
      \begin{align}\label{null_set_count_along_test_plan}
         \int_0^1 \int_{C([0,1], M)} \mathbbm{1}_{N}(\gamma_t)\,\di\boldsymbol{\pi}(\gamma)\di t =  \int_{C([0,1], M)} \int_0^1 \mathbbm{1}_{N}(\gamma_t)\,\di t\di\boldsymbol{\pi}(\gamma) \geq \frac{\e}{n}>0.
     \end{align}
     By the definition of a test plan, we have that for all $t\in [0,1]$ it holds $(e_t)_\#\boldsymbol{\pi} \ll \vol_g$.
     As $N$ is a $\di\vol_g$-null set, we get that for all $t\in [0,1]$, it holds
     \begin{align*}
         \int_{M} \mathbbm{1}_{N}(y)\, \di (e_t)_\#\boldsymbol{\pi}(y) =0.
     \end{align*}
     That gives
     \begin{align*}
          \int_0^1 \int_{C([0,1], M)} \mathbbm{1}_{N}(\gamma_t)\,\di\boldsymbol{\pi}(\gamma)\di t = \int_0^1 \int_M \mathbbm{1}_{N}(y)\,\di(e_t)_\#\boldsymbol{\pi}(y)\di t =0,
     \end{align*}
     which contradicts \eqref{null_set_count_along_test_plan}. This proves the Lemma.
 \end{proof}
We conclude this section with the definition of a quasi-Riemannian manifold proposed by De Giorgi in \cite{de1990conversazioni}, that we will refer to later:
\smallskip

\begin{definition}\label{def:quasi_riemannian_manifold}
    Given a manifold $M$, a distance $\sfd$ on $M$, and a measurable positive definite inner product $g$ on $TM$, we  
  call  $(M, \sfd, g)$ a quasi-Riemannian manifold if there exists an atlas $\mathcal{A} = \{W, \psi\}$ such that the following conditions hold:
    \begin{itemize}
        \item[(i)]  In each $(W, \psi)$, there exist constants $\lambda, \Lambda > 0$ such that for any $x, y \in W$, it holds
        \begin{align*}
            \lambda \leq (\psi_*g)^{ij}\partial_{x^i}\sfd^W_{euc}(x, y)\partial_{x^j}\sfd^W_{euc}(x, y) \leq \Lambda.
        \end{align*}
        \item[(ii)] For any $f \in C^0_c(W)$, it holds \begin{align*}
            \int_W f \di \mathcal{H}^d_{\sfd} = \int_W f\sqrt{\det g}\,  \di \mathcal{L}^d,
        \end{align*}
        where $\mathcal{H}^d_{\sfd}$ denotes the $d$-dimensional Hausdorff measure with respect to the distance $\sfd$.
        \item[(iii)]  For any Lipschitz continuous $f: M \to \R$, 
 and almost every $x \in M$ it holds 
        \begin{align*}
            \limsup_{y \to x} \frac{|f(x)-f(y)|}{\sfd(x, y)} = |\nabla_g f|_g(x).
        \end{align*}
    \end{itemize}
\end{definition}
\section{Construction of examples}\label{sec:4}
In this section, we will construct Riemannian manifolds for with non-continuous metrics that together with the distance $\sfd_g$ as defined in the previous section, do not give rise to quasi-Riemannian manifolds or infinitesimally Hilbertian metric measure spaces.  
\subsection{An introductory example of a non-quasi Riemannian manifold}\label{subsec:first_counter_quasi_riem}
We will now construct a metric $g$ on $\R^d$ for $d \geq 2$ such that $g, g^{-1} \in L^\infty$ and such that there exists a smooth function $f:\R^d \to \R$ such that $|Df|_w,|Df|> |\nabla_g f|_g$ on a set of positive measure, which implies that $(M, \sfd_g, g)$, where $\sfd_g$ is defined as in \eqref{null_set_def_of_metric}, is not a quasi-Riemannian manifold.
In \cite{sturm1997diffusion}, Sturm proved that given a Riemannian manifold $(M, g)$, $g, g^{-1} \in L^\infty_{\rm loc}$, there exists a metric $g'$ on $M$ such that $g', {g'}^{-1} \in L^\infty_{\rm loc}$, $g'<g$ on a set of positive $\vol_g$-measure and $\sfd_g= \sfd_{g'}$, which immediately shows that Questions \ref{q1} and \ref{q2} may have negative answers. The example presented in this subsection is related to the one in \cite{de1995lip} and specifically designed to pave the way for examples, where also infinitesimal Hilbertianity fails, that will be discussed in later subsections

Let $q_i \in \Q$ be a counting of $\Q$. Pick a $\kappa \in (0, \frac{1}{8})$ and construct the following open set 
\begin{align}\label{ex1_rational_cover}
    O:= \bigcup_{i=1}^\infty \Big(q_i-\frac{\kappa}{2^i}, q_i + \frac{\kappa}{2^i}\Big).
\end{align}
Then $\mathcal{L}^1(O) \leq 2\kappa$ and $O$ is dense in $\R$. We now define
\begin{align*}
    &\theta:= (2 -\mathbbm{1}_{O}) \in L^\infty(\R), \\
    &g(x_1, \ldots, x_d):= \theta(x_1)\mathrm{Id} \in L^\infty(\R^d, \R^{d\times d}).
\end{align*}
Then $(\R^d, g)$ satisfies Assumption \ref{assumption_on_mfd} and $(\R^d, \sfd_g, \vol_g)$ satisfies condition \eqref{condition_on_measure_bounded_compression}.
We observe that  
\begin{align*}
    (2-2\kappa) \leq \int_{(0,1)} \theta \, \di \mathcal{L}^1 &\leq \frac{3}{2} \mathcal{L}^1\Big(\big\{\theta < \frac{3}{2}\big\}\cap(0,1)\Big) + 2 \mathcal{L}^1\Big(\big\{\theta \geq \frac{3}{2}\big\}\cap (0,1)\Big) \\
    &=\frac{3}{2}\Big(1-\mathcal{L}^1\Big(\big\{\theta \geq \frac{3}{2}\big\}\cap (0,1)\Big)\Big)+ 2\mathcal{L}^1\Big(\big\{\theta \geq \frac{3}{2}\big\}\cap (0,1)\Big). 
\end{align*}
It follows that
\begin{align}\label{big_values_of_g}
    \mathcal{L}^1\Big(\big\{\theta \geq \frac{3}{2}\big\}\cap (0,1)\Big) \geq 1-4\kappa \geq \frac{1}{2}.
\end{align}
Now for any point $y \in  (0,1) \times\R^{d-1}$ consider the curve $\gamma^y:[0,1]\to \R^d, t \mapsto y+te_2$. 

\textbf{Claim.}
For all $t \in (0,1)$, $h \in (0, 1-t)$ it holds
\begin{align} \label{distance_along_straight_curves}
    \sfd_g(\gamma^y_t, \gamma^y_{t+h})=h. 
\end{align}
\textit{Proof of the claim.} Fix $t, h$ as above and fix an $\e>0$. Note that from $g \geq \mathrm{Id}$, we get that $\sfd_g \geq \sfd_{euc}$, hence $\sfd_g(\gamma^y_t, \gamma^y_{t+h})\geq h$. Thus, it remains to prove that $\sfd_g(\gamma^y_t, \gamma^y_{t+h}) \leq h$.   There exists a $q \in \Q\cap (0,1)$ such that $|x_1(y)-q|\leq \e$. Let $i\in \N$ be such that $q_i = q$ in the chosen counting.  Then define the curves
\begin{align*}
    &\lambda_1:[0, 1] \to \R^d, s \mapsto  \gamma^y_t + s(q-x_1(y))e_1, \\
    &\lambda_2:[0,h]\to \R^d, s \mapsto \gamma^y_t + (q-x_1(y))e_1 + se_2, \\
    &\lambda_3:[0, 1] \to \R^d, s \mapsto \gamma^y_t + (q-x_1(y))e_1 + he_2 -s(q-x_1(y))e_1. 
\end{align*}
Denote by $\lambda$ the concatenation of those three curves and note that it yields a piecewise smooth curve that connects $\gamma^y_t$ and $\gamma^y_{t+h}$.
Now, for all $\varsigma>0$, it holds $g_\varsigma = g * \rho_\varsigma \leq 2\mathrm{Id}$, so we get that $L_{g_\varsigma}(\lambda_1)+ L_{g_\varsigma}(\lambda_3) \leq 4\e$.  
For $\varsigma \leq \frac{\kappa}{2^{i+2}}$ that $g_\varsigma|_{\{x_1=q\}}\equiv \mathrm{Id}$, hence $L_{g_\varsigma}(\lambda_2) =h$. 
Thus, 
\begin{align*}
    \lim_{\varsigma \to 0} \sfd_{g_\varsigma}(\gamma^y_t, \gamma^y_{t+h}) \leq \lim_{\varsigma \to 0} L_{g_\varsigma}(\lambda) \leq h+4\e.
\end{align*}
As $\e$ was arbitrary, this proves the claim. \\

By \eqref{big_values_of_g}, we can choose a set $E'\subset \{\theta \geq \frac{3}{2}\big\}\cap (0,1)$ such that $\mathcal{L}^1(E')>0$. Denote $E:= E' \times (0,1)^{d-1}$. Define the test plan $\boldsymbol{\pi} \in \mathcal{P}(C^\infty((0,1), \R^d))$ via
\begin{align*}
    \di\boldsymbol{\pi}(\gamma):= \left\{ \begin{array}{ll}
           \frac{1}{\mathcal{L}^d(E)}\di\mathcal{L}^d(y) & \mathrm{if}\ \gamma = \gamma^{y}, \ \mathrm{for \ some} \ y \in  E, \\
           0\  & \mathrm{otherwise}.
        \end{array}\right.
\end{align*}
Note that this is indeed a feasible test plan, as $\di\mathcal{L}^d \leq \di\vol_g \leq 2^{\frac{d}{2}}\di\mathcal{L}^d$.
But now for all $y\in E$, almost all $t\in [0,1]$, we get that $g(\gamma^y_t) \geq \frac{3}{2}\mathrm{Id}$ and $\dot{\gamma}^y_t = e_2$ and $|\dot{\gamma}^y_t|_g \geq \sqrt{\frac{3}{2}}$. But by \eqref{distance_along_straight_curves}, we get that the metric speed $|\dot{\gamma}^y_t| =1$ for all $y, t$, hence we found a test plan that gives positive measure to curves whose metric speed and the norm of the derivative are not equal. 

We will now give an example of a function $f\in C^\infty(\R^d)$ for which the norm of the gradient is strictly smaller than the minimal weak upper gradient and the metric slope on a set of points with positive measure. The function is given by $f: \R^d\to \R, x=(x_1, \ldots, x_d) \mapsto x_2$. Then for all $y\in E$, we have that $|\nabla_g f|_g(y) \leq \sqrt{\frac{2}{3}}$. But now for any point $y \in \R^d$, we have that 
\begin{align}
    |Df|(y) = \limsup_{z \to y} \frac{|f(z)-f(y)|}{\sfd_g(z, y)} \geq \lim_{h \to 0} \frac{|f(y+he_2)-f(y)|}{\sfd_g(y+he_2, y)} = \lim_{h \to 0} \frac{|h|}{|h|} = 1> \sqrt{\frac{2}{3}}.
\end{align}
This shows that for $y \in E' \times \R^{d-1}$, it holds $|Df|(y) > |\nabla_g f|_g(y)$. 

Finally, the test plan $\boldsymbol{\pi}$ shows, that $|Df|_w > |\nabla_g f|_g$ on a set of positive measure. Indeed, we note that 
\begin{align*}
    \int_{C([0,1], \R^d)} |f(\gamma_1)-f(\gamma_0)| \,\di \boldsymbol{\pi}(\gamma) = \int_{E} \frac{1}{\mathcal{L}^d(E)} \,\di \mathcal{L}^d = 1.
\end{align*}
On the other hand, we have that 
\begin{align*}
    \int_{C([0,1], \R^d)} \int_0^1 |\nabla_g f(\gamma_t)|_g |\dot{\gamma}_t| \, \di t \di \boldsymbol{\pi}(\gamma) \leq  \int_{E} \int_0^1 \sqrt{\frac{2}{3}} \frac{1}{\mathcal{L}^d(E)}\, \di t \di \mathcal{L}^d = \sqrt{\frac{2}{3}}.
\end{align*}
Hence, we have proved the following:
\smallskip

\begin{proposition}\label{first_counterexample}
    For $d \geq 2$, there exists a Riemannian metric $g$ on $\R^d$ such that $g, g^{-1} \in L^\infty$ and such that 
    \begin{itemize}
        \item[(i)] There exists a test plan that gives positive measure to curves for which the metric speed is strictly smaller than the norm of the derivative for almost every time $t \in [0,1]$. 
        \item[(ii)] There exists a function $f\in C^1(\R^d)$ for which the minimal weak upper gradient and the metric slope are strictly larger than the norm of the gradient on a set of positive measure, i.e., $|Df|, |Df|_w > |\nabla_g f|_g$ on a set of positive $\vol_g$-measure.
    \end{itemize}
\end{proposition}
\subsection{An introductory example of non-infinitesimal Hilbertianity}\label{subsec:first_counter_inf_hilbertian}
In this subsection, we will more closely analyse the infinitesimal structure of a manifold similar to the previous example. 
This following example metric resembles the previous one, with the sole difference, that we only perturb the metric inside the interior of the unit cube $(0,1)^d\subset\R^d$ for $d \geq 2$ and leave it constant outside that set. We will prove that the resulting metric measure space is not infinitesimally Hilbertian.

Let $q_i \in \Q$ be a counting of $\Q \cap (0,1)$,  $\kappa \in (0, \frac{1}{8})$ and define 
\begin{align}\label{ex1_rational_cover_infH}
    O:= \bigcup_{i=1}^\infty \Big(q_i-\frac{\kappa}{2^i}, q_i + \frac{\kappa}{2^i}\Big).
\end{align}
Then $\mathcal{L}^1(O) \leq 2\kappa$ and $O$ is dense in $(0,1)$. Define
\begin{align*}
    &\theta:= (2 -\mathbbm{1}_{O}) \in L^\infty((0,1)), \\
    &g(x_1, \ldots, x_d):= 2\mathbbm{1}_{((0,1)^d)^c}\mathrm{Id} + \mathbbm{1}_{(0,1)^d}\theta(x_1)\mathrm{Id} \in L^\infty(\R^d, \R^{d\times d}).
\end{align*}
Note that $(\R^d, g)$ satisfies Assumption \ref{assumption_on_mfd}, $(\R^d, \sfd_g, \vol_g)$ satisfies \eqref{condition_on_measure_bounded_compression}, and $g$ is constant outside the precompact set $(0,1)^d$. Inside $(0,1)^d$, $g$ the metric space arising from $(M,g)$ will behave just as in the previous subsection, we only restrict this behaviour to a compact set to conclude the lack of infinitesimal Hilbertianity. 

To see that we will first quantitatively determine the difference between the metric speed of curves and the norms of their derivatives.

We start with a preparatory lemma:
\smallskip

\begin{lemma}\label{derivative_truncation}
    Let $f:[0,1] \to \R$ be a Lipschitz continuous function and $a < b \in \R$. Then 
    \begin{align}
        \int_{f^{-1}((a,b))} f' \, \di t \in [a-b, b-a].
    \end{align}
\end{lemma}
\begin{proof}
    We define the function $h:[0,1] \to \R$ as 
    \begin{align*}
        h(t):= \max(a, \min(f(t), b)).
    \end{align*}
    Then $h$ is Lipschitz continuous and admits a weak derivative $h'$. Moreover, in the open set $f^{-1}((a,b)) \subset [0,1]$, we have that $f=h$, hence $h' = f'$ almost everywhere in $f^{-1}((a,b))$.
    Now, on $f^{-1}((-\infty,a])$ we have that $h \equiv a$. 
    
    \textbf{Claim.} For almost every $t \in f^{-1}((-\infty,a])$, it holds $h' =0$. \\
    \textit{Proof of the claim.} Indeed, choose a point $t \in [0,1]$ that is a differentiability point of $h$ and that is a Lebesgue point of both $f^{-1}((-\infty,a])$ and $h'$. Suppose for a contradiction that $h'(t) \neq 0$. 
    Then there exists a neighbourhood $U \ni t$ such that $h(s) \neq h(t)$ for every $s \in U \setminus \{t\}$.
    But as $t$ is a Lebesgue point of $f^{-1}((-\infty,a])$, we have that $(f^{-1}((-\infty,a]) \cap U) \setminus \{t\} \neq \emptyset$. Hence, there exists a point in $s \in f^{-1}((-\infty,a])$ such that $h(s) \neq h(t)$, which contradicts the definition of $h$. As those Lebesgue points have full measure, this proves the claim.

    We similarly get that on $f^{-1}([b, \infty))$, it holds $h' \equiv 0$. Then
    \begin{align}
      \int_{f^{-1}((a,b))} f' \, \di t = \int_0^1 h' \, \di t = h(1)-h(0).  
    \end{align}
    The lemma follows as $h(0), h(1) \in [a,b]$.
\end{proof}

With Proposition \ref{metric_speed_depends_on_direction}, we will now determine the metric speed of a curve through a point $y \in \R^d$ depending on the derivative.

First suppose that $y \in (0,1)^d$ and that $y$ is a Lebesgue point of $\{g=2\}$. 
In this argument, we will work with the definition of the metric $\sfd_g$ via \eqref{null_set_def_of_metric}. 
We will at first determine the shape and length of an almost distance minimising curve between two points that are close to each other. 
Fix an $\e >0$. Then, by the Lebesgue differentiation theorem, there exists a $\delta > 0$ such that for each open interval $I \subset \R$ with $x_1(y) \in I$, and $\mathcal{L}^1(I) \leq \delta$, it holds
\begin{align}\label{delta:lebesgue_point_condition_ex1/4.2}
    \frac{1}{\mathcal{L}^1(I)}\int_I |\theta-2| \, \di \mathcal{L}^1 \leq \e. 
\end{align}

Fix the null set $N^{(1)}_g \subset \R$ as the set of non-Lebesgue points of $\theta$ and define $N_g := N^{(1)}_g \times \R^{d-1}$. Note that then $O \cap N^{(1)}_g = \emptyset$. 

Let $z \in (0,1)^d$ be such that $\sfd_{euc}(y, z) \leq \frac{1}{8}\dist_{euc}(y, \partial (0,1)^d)$. Then any $2$-shortest curve connecting them, i.e., any curve $\gamma:[0,1]\to \R^d$ from $y$ to $z$ such that $\gamma \perp N_g$ and $L_g(\gamma) \leq 2 \sfd_g(y, z)$, has to lie inside $(0,1)^d$. Let $\gamma:[0,1]\to (0,1)^d$ be a Lipschitz continuous curve transversal to $N_g$ such that $\gamma(0)=y$ and $\gamma(1)=z$.
We will now estimate $\sfd_g(y,z)$ depending on $z-y \in \R^d$. 
We assume that $x_1(z) \geq x_1(y)$ and $|x_1(z)-x_1(y)| < \delta$. The case $x_1(z) \leq x_1(y)$ is analogous.

\textbf{Step 1:} Reduction to curves with uniformly bounded $x_1$ coordinates.\\
Pick $a_y \leq x_1(y)$ and $a_{z} \geq x_1(z)$ such that $a_y, a_{z} \in O\cap (0,1)$, and
\begin{align}\label{truncation_interval}
    a_{z}-a_y \leq \delta, \quad x_1(y)-a_y \leq \e,\quad  a_{z}-x_1(z) \leq \e,\quad  a_{z}-a_y \leq 2|z-y|_{euc}. 
\end{align}
Now define $\Tilde{\gamma}:[0,1] \to \R^d$ via 
\begin{align}
    &x_1(\Tilde{\gamma})(t) = \max(a_y, \min( a_{z}, x_1(\gamma_t))), \nonumber \\
    &x_i(\Tilde{\gamma})(t) = x_i({\gamma}_t)\ \mathrm{for} \ i=2, \ldots, d.
\end{align}
Then $\Tilde{\gamma}$ is Lipschitz continuous and transversal to $N_g$. Indeed, on $(x_1 \circ \gamma)^{-1}((-\infty, a_y] \cup [a_{z}, \infty))$, we get that $\Tilde{\gamma} \in \{a_y, a_{z}\} \times (0,1)^{d-1} \subset N^c_g \cap \{g \equiv \mathrm{Id}\}$ and on $(x_1 \circ \gamma)^{-1}((a_y, a_{z}))$, we have that $\gamma = \Tilde{\gamma}$. Note that the methods from the proof of Lemma \ref{derivative_truncation} yield that
\begin{align}
 |\dot{\Tilde{\gamma}}_t|_{euc} \leq |\dot{\gamma}_t|_{euc} \ \mathrm{a.e.\ on} \ [0,1].
\end{align}
Now, we can compute 
\begin{align}\label{estimate_length_gamma_1}
    L_g(\gamma) &= \int_{[0,1]} \sqrt{\theta(x_1(\gamma_t))} |\dot{\gamma}_t|_{euc} \, \di t \nonumber \\
    &= \int_{(x_1 \circ \gamma)^{-1}((a_y, a_{z}))} \sqrt{\theta(x_1(\gamma_t))} |\dot{\gamma}_t|_{euc} \, \di t+ \int_{(x_1 \circ \gamma)^{-1}((-\infty, a_y]\cup [ a_{z}, \infty))} \sqrt{\theta(x_1(\gamma_t))} |\dot{\gamma}_t|_{euc} \, \di t \nonumber \\
    & \geq \int_{(x_1 \circ \gamma)^{-1}((a_y, a_{z}))} \sqrt{\theta(x_1(\gamma_t))} |\dot{\gamma}_t|_{euc} \, \di t+ \int_{(x_1 \circ \gamma)^{-1}((-\infty, a_y]\cup [ a_{z}, \infty))} |\dot{\Tilde{\gamma}}_t|_{euc} \, \di t\nonumber \\
    & = \int_{(x_1 \circ \gamma)^{-1}((a_y, a_{z}))} \sqrt{\theta(x_1(\Tilde{\gamma}_t))} |\dot{\Tilde{\gamma}}_t|_{euc} \, \di t+ \int_{(x_1 \circ \gamma)^{-1}((-\infty, a_y]\cup [ a_{z}, \infty))} \sqrt{\theta(x_1(\Tilde{\gamma}_t))} |\dot{\Tilde{\gamma}}_t|_{euc} \, \di t \nonumber\\
    &= \int_{[0,1]} \sqrt{\theta(x_1(\Tilde{\gamma}_t))} |\dot{\gamma}_t|_{euc} \, \di t=L_g(\Tilde{\gamma}).
\end{align}
Hence, we found a curve $\Tilde{\gamma}$ with the same endpoints as $\gamma$ that is transversal to $N_g$ with bounded $x_1$ coordinates that is at most as long as $\gamma$. From now on, we write $\gamma:= \Tilde{\gamma}$ with a slight abuse of notation. 

\textbf{Step 2.} Reduction to the length of curves consisting of two straight lines. \\
Define
\begin{align*}
    &v_{(1)}:= \int_{(\theta \circ x_1 \circ \gamma)^{-1}(1)}\dot{\gamma}_t \, \di t \in \R^d,\\
    &v_{(2)}:= \int_{(\theta \circ x_1 \circ \gamma)^{-1}(2)}\dot{\gamma}_t \, \di t \in \R^d.
\end{align*}
Note that $v_{(1)} + v_{(2)} = {z}-y$ and by classical results about the Euclidean space it holds  
\begin{align}\label{estimate_length_of_tilde_gamma_2}
    L_g({\gamma}) &= \int_{[0,1]} \sqrt{\theta(x_1({\gamma}_t))} |\dot{\gamma}_t|_{euc} \, \di t \nonumber\\
    & = \int_{(\theta \circ x_1 \circ \gamma)^{-1}(2)} \sqrt{2}|\dot{\gamma}_t|_{euc} \, \di t + \int_{(\theta \circ x_1 \circ \gamma)^{-1}(1)}|\dot{\gamma}_t|_{euc} \, \di t  \nonumber\\
    & \geq \sqrt{2}|v_{(2)}|_{euc} + |v_{(1)}|_{euc}. 
\end{align}
Write $O \cap (a_y, a_{z}):= \bigcup_{n\in \N} I_n$, where $I_n$ are pairwise disjoint open intervals. By our choice of $\delta$ (see \eqref{delta:lebesgue_point_condition_ex1/4.2} and \eqref{truncation_interval}), we get that $\mathcal{L}^1(O \cap (a_y, a_{z})) \leq \e(a_{z}-a_y)$.
Now, by Lemma \ref{derivative_truncation}, we get the following estimate on the first component $x_1(v_{(1)})$ of $v_{(1)} = \sum_{i=1}^d x_i(v_{(1)}) e_i$. 
\begin{align*}
    |x_1(v_{(1)}) |=\Big| \sum_n \int_{(x_1 \circ \gamma)^{-1}(I_n)} \dot{(x_1 \circ \gamma)}(t) \di t \big| \leq \sum_n \mathcal{L}^1(I_n) \leq \e|a_{z}-a_y| \leq 2\e|z-y|_{euc}.
\end{align*}
Now define 
\begin{align*}
    w_{(1)}:= v_{(1)}- x_1(v_{(1)}) e_1 \in \{0\} \times \R^{d-1}, \ w_{(2)} = v_{(2)}+ x_1(v_{(1)}) e_1 \in \R^d. 
\end{align*}
Then $w_{(1)}+w_{(2)}= z-y$ and 
\begin{align*}
    |w_{(1)}-v_{(1)}|_{euc}, |w_{(2)}-v_{(2)}|_{euc} \leq 2\e|z-y|_{euc}. 
\end{align*}
Hence 
\begin{align}\label{length_estimate_w_vectors}
     L_g({\gamma}) \geq \sqrt{2}|w_{(2)}|_{euc} + |w_{(1)}|_{euc}- 6\e|z-y|_{euc},
\end{align}
where $w_{(1)}+w_{(2)}= z-y$ and $w_{(1)} \in \mathrm{span}(e_2, \ldots e_d)$. 

\textbf{Step 3.} Optimisation of $w_{(1)}$ and $w_{(2)}$ under the above constraints. \\
By potentially applying a linear transform from $\mathrm{SO}(d-1)$ that rotates the last $d-1$ entries, we may assume that $z-y \in \mathrm{span}(e_1, e_2)$.
We reduce the right hand side of \eqref{length_estimate_w_vectors}, if we orthogonally project both $w_{(1)}$ and $w_{(2)}$ onto $\mathrm{span}(e_1, e_2)$, so we may assume that  $w_{(1)},w_{(2)} \in \mathrm{span}(e_1, e_2)$. 
Now, if $z-y= ae_1$, for some $a \in \R$ we get that $w_{(1)}=0$ and hence
\begin{align}\label{estimate_x_1_direction_eps}
    L_g({\gamma}) \geq \sqrt{2}a -6a\e.
\end{align}
Otherwise, we may write $z-y = a(be_1+e_2)$, where $a, b \in \R$, $a \neq 0$.
Write $w_{(1)} = a(1-c)e_2$ and $w_{(2)} = a(be_1+ ce_2)$.
Now, we want to choose the optimal parameter $c \in \R$ such that
\begin{align*}
    \sqrt{2}|w_{(2)}|_{euc} + |w_{(1)}|_{euc},
\end{align*}
that is \eqref{length_estimate_w_vectors} without the error term, becomes minimal under the constraint that $w_{(1)} \in \mathrm{span}(e_2)$.
The error term in \eqref{length_estimate_w_vectors} is independent of $c$, hence this reduces to finding $c \in \R$ such that 
\begin{align}\label{one_dim_miminisation_problem}
    f_b(c):=|1-c| + \sqrt{2}\sqrt{b^2+c^2}
\end{align}
is minimised. From the definition of the function one can see that for $c \leq 0$ it holds $f_b(c) \geq f_b(0)$. Similarly, if $c \geq 1$, we get that $f_b(c) \geq f_b(1)$. Hence we only need to minimise for $c \in (0,1)$, in which $f_b$ is continuously differentiable and check the boundary values.
We have that 
\begin{align*}
    f_b(0) = 1+\sqrt{2}|b|, f_b(1) = \sqrt{2(1+b^2)}.
\end{align*}
Now the derivative for $c \in (0,1)$ is given by
\begin{align*}
    \frac{\di}{\di c}f_b(c) = -1 +\frac{\sqrt{2}c}{\sqrt{b^2+c^2}}.
\end{align*}
The derivative can only vanish if  $\sqrt{2}c_* = \sqrt{b^2+c_*^2} \iff c_*= |b|$.
Now, if $|b| \leq 1$, we get that $c_* \in [0,1]$ and 
\begin{align*}
    f_b(c_*) = 1-|b| + \sqrt{2(b^2+b^2)} = 1+|b|.
\end{align*}
Then, we get that $f_b(0), f_b(1) \geq f_b(c_*)$.
If $|b| \geq 1$, we only need to check the boundary values and get that $f_b(0) \geq f_b(1)$.
Hence, for $|b|\geq 1$ the optimal value is $\sqrt{2(b^2+1)}$ and for $|b| \leq 1$, the optimal value is $1+|b|$.

\textbf{Step 4.} Description of the metric speed of a curve depending on its derivative $u$.\\
Recall that we chose $y$ to be a Lebesgue point of $\{g=2\}$.
Writing $z-y = u = (u_1, u')$ where $u_1 \in x_1(y)+(-\delta, \delta) \subset \R$ and $u' \in \R^{d-1}$, the previous three steps yield that
\begin{align*}
    L_g(\gamma) \geq \left\{ \begin{array}{ll}
           |u_1| + |u'|_{euc} - 6\e|u|_{euc} & \mathrm{if}\ |u_1| \leq |u'|_{euc}, \\
           \sqrt{2}|u|_{euc}- 6\e|u|_{euc}\  & \mathrm{otherwise}.
        \end{array}\right.
\end{align*}
More precisely, using that $\sfd_g \geq \boldsymbol{\delta}_{N_g}$, we get that for $u$ as above,  it holds
\begin{align}\label{asymptotic_lower_bound_metric_ex1}
    \sfd_g(y, y+u) \geq \left\{ \begin{array}{ll}
           |u_1| + |u'|_{euc} - 6\sigma(y, u_1)|u|_{euc} & \mathrm{if}\ |u_1| \leq |u'|_{euc}, \\
           \sqrt{2}|u|_{euc}- 6\sigma(y, u_1)|u|_{euc}\  & \mathrm{otherwise},
        \end{array}\right.
\end{align}
where 
\begin{align}
    \sigma(y, u_1) := \inf\{\e>0, \Big|\frac{1}{|I|}\int_I \theta \, \di x-2\Big| \leq \e, \forall I\subset \R\ \mathrm{open}, x_1(y) \in I, \mathcal{L}^1(I)< |u_1|\}.
\end{align}
As $x_1(y)$ is chosen as a Lebesgue point of $\theta$, we get that $\sigma(y, u_1) \to 0$ as $|u_1|\to 0$. 

We will next prove an upper bound on $\sfd_g(y,z)$ by working with the regularisation, i.e., smooth metrics $g_\epsilon$ that converge to $g$ almost everywhere and such that $\sfd_{g_\epsilon} \to \sfd_g$. We will construct curves $\hat{\gamma}$ that are given by the concatenation of four straight lines, two of which will be arbitrarily short. The remaining two lines will be optimal in the sense of steps 2 and 3. 
Fix $\epsilon >0$. For any $\varsigma >0$ and $g_\varsigma := \rho_{\varsigma} * g$, it holds $g_\varsigma \leq 2 \mathrm{Id}$. Hence, for all $\varsigma>0$, we have that
\begin{align*}
    \sfd_{g_\varsigma}(y,z) \leq \sqrt{2}|u|_{euc}.
\end{align*}
Moreover, we can find a rational $q \in \Q\cap (0,1)$ such that $|q-x_1(y)| \leq \epsilon|u|_{euc}$. 
Fix a $0<\alpha < \beta <1$ such that $U:= (\alpha,\beta)^{d-1}$ satisfies that $y, z \in (0,1) \times U$. 
Now for some $\varsigma_\epsilon >0$ small enough, we have that for $\varsigma \in (0, \varsigma_\epsilon)$, it holds $g_\varsigma =\mathrm{Id}$ on $\{q\} \times U$. If $|u_1|< |u'|_{euc}$, write $b= \frac{u_1}{|u'|_{euc}}$. Then define the curves
\begin{align*}
    &\gamma_1:[0,1] \to \R^d, t \mapsto y+t(q-x_1(y))e_1, \\
    &\gamma_2:[0,1] \to \R^d, t \mapsto y+(q-x_1(y))e_1 + t(1-|b|)u', \\
    &\gamma_3:[0,1] \to \R^d, t \mapsto  y+(q-x_1(y))e_1 + (1-|b|)u'-t(q-x_1(y))e_1, \\
    &\gamma_4:[0,1] \to \R^d, t \mapsto  y + (1-|b|)u' + t|b|u'+tu_1e_1.
\end{align*}
We then get that 
\begin{align*}
    L_{g_\varsigma}(\gamma_1), L_{g_\varsigma}(\gamma_3) \leq \sqrt{2}|q-x_1(y)| \leq \sqrt{2}\epsilon|u|_{euc}.
\end{align*}
Moreover as $g_\varsigma \equiv \mathrm{Id}$ on $\{q\} \times U \supset \gamma_2([0,1])$, we get that 
\begin{align*}
     L_{g_\varsigma}(\gamma_2) = (1-|b|)|u'|_{euc}.
\end{align*}
Finally, we get that 
\begin{align*}
 L_{g_\varsigma}(\gamma_4) \leq \sqrt{2}\big|bu'+ u_1e_1\big|_{euc}= 2|u_1|.
\end{align*}
Hence, for all $\varsigma$ small enough, the concatenation $\hat{\gamma}$ of the four curves satisfies
\begin{align}
    \sfd_{g_\varsigma}(y, z) \leq L_{g_\varsigma}(\hat{\gamma}) \leq 3\epsilon|u|_{euc} + \frac{|u'|_{euc}-|u_1|}{|u'|_{euc}}|u'|_{euc} + 2|u_1| = |u_1| + |u'|_{euc} +3\epsilon|u|_{euc}.
\end{align}
Hence, using $\sfd_g = \lim_{\varsigma \to 0} \sfd_{g_\varsigma}$ and the fact that $\epsilon$ was independent of $|u_1|$, sending $\epsilon \to 0$ yields that 
\begin{align}\label{upper_bound_metric_ex1}
    \sfd_g(y, y+u) \leq \alpha(u):= \left\{ \begin{array}{ll}
           |u_1| + |u'|_{euc}  & \mathrm{if}\ |u_1| \leq |u'|_{euc}, \\
           \sqrt{2}|u|_{euc}\  & \mathrm{otherwise}.
        \end{array}\right.
\end{align}

Now considering the curve $\gamma^{y,u}:[0,1] \to \R^d, t \mapsto y+tu$, for $u \in \R^d$, \eqref{upper_bound_metric_ex1} and \eqref{asymptotic_lower_bound_metric_ex1} give that 
\begin{align*}
    0 \leq \lim_{h \to 0} \frac{\alpha(hu)-\sfd_g(\gamma^{y,u}_0, \gamma^{y,u}_h)}{|h|} \leq \lim_{h \to 0} \frac{\sigma(y, hu_1)|hu|_{euc}}{|h|} = \lim_{h \to 0} \sigma(y, hu_1)|u|_{euc} =0.
\end{align*}
Hence
\begin{align*}
    \lim_{h \to 0} \frac{\sfd_g(\gamma^{y,u}_0, \gamma^{y,u}_h)}{|h|} = \lim_{h \to 0} \frac{\alpha(hu)}{|h|} = \alpha(u).
\end{align*}
This completes the description of the metric speed of curves through points $y$ that are Lebesgue points of $\{g =2\}$ depending on the derivative $u$. This can be seen as the explicit expression of what is called the ``derivative" of the distance function in \cite{de1995lip}.

For any point $y \in (0,1)^d$ such that $x_1(y) \in O$, we have that $g \equiv \mathrm{Id}$ in a neighbourhood of $y$. Let $\gamma:[0,1] \to M$ be an absolutely continuous curve through $y$ and fix $t \in [0,1]$ such that $\gamma_t =y$. Suppose that $\gamma$ is differentiable at time $t$. The metric speed $|\dot{\gamma}_t|$ of $\gamma$ at time $t$  depends only on the values of $\gamma$ at an arbitrarily small neighbourhood $I \subset [0,1]$ of $t$. The continuity of $\gamma$ allows to assume $\gamma(I) \subset O \times \R^{d-1}$. As $g$ is constant, and hence smooth, in $O \times \R^{d-1}$, the results on smooth Riemannian metrics (see also \cite[Proposition 4.10]{burtscher2012length}) show that in this case $|\dot{\gamma}_t| = |\dot{\gamma}_t|_{g(y)} = |\dot{\gamma}_t|_{euc}$. 

Similarly, for every $y \in ([0,1]^d)^c$, we have that $g \equiv 2 \mathrm{Id}$ in a neighbourhood of $y$. Hence, for every absolutely continuous curve $\gamma$ through $y$ with $t$ as above, we have that $|\dot{\gamma}_t| = |\dot{\gamma}_t|_{g(y)} = \sqrt{2}|\dot{\gamma}_t|_{euc}$.

Now note that for $\mathcal{L}^1$-almost all $x \in \R$, it holds that $x \in O$ or $x$ is a Lebesgue point of $\{\theta =2\}$. 
Write $N^{(2)} \subset \R$ for the null set of points satisfying neither of that and define the $\mathcal{L}^d$-null set $N_2 := N^{(2)} \times \R^{d-1} \cup (\partial (0,1)^d)$.
In the spirit of Proposition \ref{metric_speed_depends_on_direction}, we have now for almost all points $y$, more precisely all points $y \in \R^d \setminus N_2$ characterised the metric speed of a curve through the point $y$ depending on its derivative at $y$. In particular, we have shown that for almost every $y$, we have that if an absolutely continuous curve $\gamma:[0,1] \to \R^d$ passes through $y$ (say at time $t \in [0,1]$) and is differentiable at $t$, then the metric speed $|\dot{\gamma_t}|$ exists and is given by $|\dot{\gamma}_t|_{euc}$ if $y \in (0,1)^d$ and $x_1(y) \in O$, $\alpha(\dot{\gamma}_t)$ if $y \in (0,1)^d$ and $x_1(y) \in \R \setminus (N^{(2)} \cup O)$, and $\sqrt{2}|\dot{\gamma}_t|_{euc}$ if $y \in ([0,1]^{d})^c$.  

\textbf{Step 5.} Example of functions that shows that the Cheeger energy does not satisfy the parallelogram identity.

This is the final step of the argument.
Consider the smooth functions $f_1, f_2, f_3, f_4: M = \R^d \to \R$ given by
\begin{align*}
    f_1: x \mapsto x_1, f_2: x \mapsto x_2, f_3 := f_1+f_2, \ \mathrm{and} \ f_4 := f_1-f_2. 
\end{align*}
Now again pick a point $y\in (0,1)^d$ such that $x_1(y)$ is a Lebesgue point of $\{\theta=2\}$. Pick a vector $0 \neq u \in \R^d$. Denote by $u_2$ the second component of $u$. Now, we may compute
\begin{align*}
    \partial_uf_1 = u_1, \partial_uf_2 = u_2, \partial_uf_3 = u_1+u_2, \partial_u f_4 = u_1-u_2.
\end{align*}
Then \begin{align*}
    \lim_{h \to 0} \frac{|f_i(\gamma^{y,u}_h) -f_i(\gamma^{y,u}_0)|}{\sfd_g(\gamma^{y,u}_0, \gamma^{y,u}_h)} = \frac{|\partial_u f_i|(\gamma_0^{y,u})}{\alpha(u)}, \ \mathrm{ for\ } 1 \leq i \leq 4.
\end{align*}
Now if $|u_1|\geq |u'|_{euc}$, we have that $\alpha(u) = \sqrt{2}|u|_{euc}$, hence
\begin{align}
    &\frac{|\partial_u f_1|}{\alpha(u)} = \frac{|u_1|}{\sqrt{2}|u|_{euc}} \leq \frac{|u_1|}{\sqrt{2}|u_1|}= \frac{1}{\sqrt{2}},\label{case_1f_1}\\
    &\frac{|\partial_u f_2|}{\alpha(u)} = \frac{|u_2|}{\sqrt{2}|u|_{euc}} \leq \frac{|u_2|}{\sqrt{2}|u_2|}= \frac{1}{\sqrt{2}},\label{case_1f_2}\\
    &\frac{|\partial_u f_3|}{\alpha(u)} = \frac{|u_1+u_2|}{\sqrt{2}|u|_{euc}} \leq \frac{|u_1|+|u_2|}{\sqrt{2(u_1^2+u_2^2)}}\leq 1,\label{case_1f_3}\\
    &\frac{|\partial_u f_4|}{\alpha(u)} = \frac{|u_1-u_2|}{\sqrt{2}|u|_{euc}} \leq \frac{|u_1|+|u_2|}{\sqrt{2(u_1^2+u_2^2)}}\leq 1\label{case_1f_4},
\end{align}
Note that \eqref{case_1f_3} and \eqref{case_1f_4} follow from the fact that the quadratic mean is greater than or equal to the arithmetic mean. Moreover, we note that the equality in \eqref{case_1f_1} can be achieved by choosing $u= u_{(1)}:=e_1$. 

If $|u_1|< |u'|_{euc}$, we have that $\alpha(u) = |u_1|+|u'|_{euc}$, hence
\begin{align}
    &\frac{|\partial_u f_1|}{\alpha(u)} = \frac{|u_1|}{|u_1|+|u'|_{euc}} \leq \frac{|u_1|}{2|u_1|+(|u'|_{euc}-|u_1|)}\leq \frac{1}{2},\label{case_2f_1}\\
    &\frac{|\partial_u f_2|}{\alpha(u)} = \frac{|u_2|}{|u_1|+|u'|_{euc}} \leq \frac{|u_2|}{|u_2|}= 1,\label{case_2f_2}\\
    &\frac{|\partial_u f_3|}{\alpha(u)} = \frac{|u_1+u_2|}{|u_1|+|u'|_{euc}} \leq \frac{|u_1|+|u_2|}{|u_1|+|u_2|}= 1,\label{case_2f_3}\\
    &\frac{|\partial_u f_4|}{\alpha(u)} = \frac{|u_1-u_2|}{|u_1|+|u'|_{euc}} \leq \frac{|u_1|+|u_2|}{|u_1|+|u_2|}= 1\label{case_2f_4},
\end{align}
Here, in \eqref{case_2f_1}, we used that $|u_1| < |u'|_{euc}$ which implies that $|u'|_{euc}-|u_1|>0$. In \eqref{case_2f_2}, \eqref{case_2f_3}, and \eqref{case_2f_4}, one can see that $u=u_{(2)}:= e_2$, $u=u_{(3)}:=e_1+e_2$, and $u=u_{(4)}:=e_1-e_2$ yield equality.
We now define the functions $G_1, G_2, G_3, G_4: \R^d \to [0, \infty]$ via
\begin{align}\label{min_weak_upper_grad_def_ex1}
    &G_1 := \frac{1}{\sqrt{2}} + \frac{\sqrt{2}-1}{\sqrt{2}}{\mathbbm{1}_{ O\times (0,1)^{d-1}}},\nonumber\\
    & G_2=\frac{1}{\sqrt{2}}+ \frac{\sqrt{2}-1}{\sqrt{2}}\mathbbm{1}_{(0,1)^d},\nonumber \\
    & G_3=G_4 = 1+(\sqrt{2}-1)\mathbbm{1}_{O \times (0,1)^{d-1}}.
\end{align}
To summarise the above definitions more easily, note that for $i=1, \ldots, 4$, we have that in $([0,1]^d)^c \cup (O \times (0,1)^{d-1})$ it holds $G_i = |\nabla_g f_i|_g$ and in $((0,1)\setminus O) \times (0,1)^{d-1}$ we have that $G_i$ equals the respective upper bound established in \eqref{case_1f_1}-\eqref{case_2f_4}.
\smallskip

\begin{proposition}\label{weak_upper_gradient_ex1}
    For $1 \leq i \leq4$, it holds that $G_i$ is a weak upper gradient of $f_i$.
\end{proposition}
\begin{proof}
    Let $\gamma:[0,1] \to M$ be absolutely continuous and transversal to a $\mathcal{L}^d$-null set $N \supset N_g \cup N_2$. As $f_i \in C^1$, we may apply the chain rule. We have that for almost every $t \in [0,1]$, it holds that $\gamma_t \in ([0,1]^d)^c \cup (O \times (0,1)^{d-1})$ or $\gamma_t \in ((0,1)\setminus O) \times (0,1)^{d-1}$ and $t$ is a differentiability point of $\gamma$.
    In the first case, we get that 
\begin{align}\label{derivative_function_along_curve_cts_area}
    \Big|\frac{\di}{\di s} f_i\circ \gamma \Big|_{s=t} \Big|= |\partial_{\dot{\gamma}_t}f_i(\gamma_t)| = |\langle \nabla_gf_i(\gamma_t), \dot{\gamma}_t\rangle_g| \leq |\nabla_g f_i|_g(\gamma_t) |\dot{\gamma}_t|_g = G_i(\gamma_t) |\dot{\gamma}_t|.
\end{align}  
In the second case, we get that either $\dot{\gamma}_t=0$, in which case $\frac{\di}{\di s} f_i\circ \gamma \big|_{s=t}=0$ or
\begin{align}\label{derivative_function_along_curve}
    \Big|\frac{\di}{\di s} f_i\circ \gamma \Big|_{s=t} \Big|= |\partial_{\dot{\gamma}_t}f_i(\gamma_t)| = \frac{|\partial_{\dot{\gamma}_t}f_i(\gamma_t)|}{\alpha(\dot{\gamma}_t)}\alpha(\dot{\gamma}_t)\leq G_i(\gamma_t)|\dot{\gamma}_t|.
\end{align}
Hence, for any transversal curve, \eqref{derivative_function_along_curve_cts_area} and \eqref{derivative_function_along_curve} together yield that  
\begin{align}
    |f_i(\gamma_1)-f_i(\gamma_0)|\leq \int_0^1 |\partial_{\dot{\gamma}_t}f_i(\gamma_t)|\,\di t  \leq \int_0^1 G_i(\gamma_t)|\dot{\gamma}_t|\, \di t,
\end{align}
where in the last step we used \eqref{upper_bound_metric_ex1} and \eqref{case_1f_1}-\eqref{case_2f_4}.
Now, we can conclude with Lemma \ref{testplans_on_transversal}.
\end{proof}
We next want to see that these are indeed the minimal weak upper gradients. For that we need the following result from \cite[Lemma 4.23]{mondino2025equivalence}. 
\smallskip

\begin{lemma}\label{Lebesguepointsadvanced_c} 
Let $k \in L^1_{\rm{loc}}(\R^d, \mathcal{L}^d)$, $x \in \R^d$, and $r > 0$.
For $\delta \in (0, r)\cap \Q$, $v \in B_r^{euc}(0)$,  we define 
    \begin{align*}
        F_{v, \delta, x}: t \to \frac{1}{\mathcal{L}^d(B_\delta(x))}\int_{B_\delta(x)} k(y + tv)\,\di\mathcal{L}^d(y), \quad \forall t, \;|t|< 1.
    \end{align*}
    Then for $\mathcal{L}^d$-a.e.\;$x \in \R^d$ and all {$\delta \in (0, r)\cap \Q$}, $v \in B_r^{euc}(0)$, we have that $t=0$ is a Lebesgue point of $ F_{v, \delta, x}$.
\end{lemma}

\begin{proposition}\label{minimal_weak_upper_gradient_ex1}
    For $1 \leq i \leq 4$, it holds that $G_i$ is the minimal weak upper gradient of $f_i$.
\end{proposition}
\begin{proof}
We will argue similarly to the proof of Proposition 4.24 in \cite{mondino2025equivalence}. 
    Fix an $i \in \{1,2,3,4\}$. We argue by contradiction. Suppose that $G_i$ is not the minimal weak upper gradient of $f_i$. Then, there exists an $\e >0$ such that the set $D_\e:= \{x: |Df_i|_w(x) < G_i(x)-\e\}$ has positive $\mathcal{L}^d$-measure. Then we can choose a point $x \in D_\e \cap N_2^c$ such that $x$ is a Lebesgue point of $|Df_i|$, $G_i$ and such that $t=0$ is a Lebesgue point of 
    \begin{align*}
         &F_{u_{(i)}, \delta, x}: t \to \frac{1}{\mathcal{L}^d(B_\delta(x))}\int_{B_\delta(x)} |Df_i|_w(y + tu_{(i)})\,\di\mathcal{L}^d(y), \ \mathrm{and} \\
         &G_{u_{(i)}, \delta, x}: t \to \frac{1}{\mathcal{L}^d(B_\delta(x))}\int_{B_\delta(x)} G_i(y + tu_{(i)})\,\di\mathcal{L}^d(y),
    \end{align*} 
    for all $\delta \in (0,1) \cap \Q$. This is possible by Lemma \ref{Lebesguepointsadvanced_c}. Then we may choose $\delta \in (0,1) \cap \Q$ small enough such that 
    \begin{align}\label{lebesgue_pt_1}
        &\frac{1}{\mathcal{L}^d(B_\delta(x))}\int_{B_\delta(x)} ||Df_i|_w(y) - |Df_i|_w(x)|\,\di\mathcal{L}^d(y) \leq \frac{\e}{8}, \ \mathrm{and}
        &\frac{1}{\mathcal{L}^d(B_\delta(x))}\int_{B_\delta(x)} |G_i(y)-G_i(x)|\,\di\mathcal{L}^d(y)  \leq \frac{\e}{8}.
    \end{align}
    Moreover, we can choose $\tau>0$ such that 
    \begin{align}\label{lebesgue_pt_2}
       & \frac{1}{2\tau} \frac{1}{\mathcal{L}^d(B_\delta(x))} \int_{-\tau}^\tau \Bigg|\int_{B_\delta(x)} |Df_i|_w(y+tu_{(i)})\,\di\mathcal{L}^d(y) - \int_{B_\delta(x)}|Df_i|_w(y) \,\di\mathcal{L}^d(y)\Bigg|\, \di t \leq \frac{\e}{8}, \ \mathrm{and} \nonumber \\
        &\frac{1}{2\tau} \frac{1}{\mathcal{L}^d(B_\delta(x))} \int_{-\tau}^\tau \Bigg|\int_{B_\delta(x)} G_i(y+tu_{(i)})\,\di\mathcal{L}^d(y) - \int_{B_\delta(x)}G_i(y) \,\di\mathcal{L}^d(y)\Bigg|\, \di t \leq \frac{\e}{8}.
    \end{align}
    Then \eqref{lebesgue_pt_1} and \eqref{lebesgue_pt_2} together with the fact that $|Df_i|_w(x) < G_i(x)-\e$, yields that
    \begin{align}
    \frac{1}{2\tau} \frac{1}{\mathcal{L}^d(B_\delta(x))} \int_{-\tau}^\tau \int_{B_\delta(x)} G_i(y+tu_{(i)})-|Df_i|_w(y+tu_{(i)})\,\di\mathcal{L}^d(y) \di t\geq \e - \frac{4\e}{8} = \frac{\e}{2}.
    \end{align}
    We now define the curve $\gamma^{y, i}:[0,1] \to \R^d$ as $\gamma^{y, i}: t \mapsto y + 2\tau(t- \frac{1}{2})u_{(i)}$. 
    Then, we define the test plan
    \begin{align}
    \di\boldsymbol{\pi}(\gamma):= \left\{ \begin{array}{ll}
           \frac{1}{\mathcal{L}^d(B_\delta(x))}\di\mathcal{L}^d(y) & \mathrm{if}\ \gamma = \gamma^{y, i}, \ \mathrm{for \ some} \ y \in  B_\delta(x), \\
           0\  & \mathrm{otherwise}.
        \end{array}\right.
    \end{align}
    Now from the fact that $\sfd_g$ and $\sfd_{euc}$ are $\sqrt{2}$-equivalent, we get that for each $y \in \R^d$ and each $t \in [0,1]$ the metric speed $|\dot{\gamma}^{y, i}| \geq \frac{1}{\sqrt{2}}|u_{(i)}|_{euc}$.  
\begin{align}\label{evaluation_test_plan_comparison}
&\int_0^1 \int_{C^0([0,1], \R^d)} (G_i(\gamma_t)-|Df_i|_w(\gamma_t))|\dot{\gamma_t}|\,\di \boldsymbol{\pi}(\gamma) \di t\nonumber \\
    &= \frac{1}{\mathcal{L}^d(B_\delta(x))} \int_{-\tau}^\tau \int_{B_\delta(x)} (G_i(y+tu_{(i)})-|Df_i|_w(y+tu_{(i)}))|\dot{\gamma}^{y,i}_t|\,\di\mathcal{L}^d(y) \di t \geq \frac{\tau\e}{\sqrt{2}}.
    \end{align}
Now we observe that for all points $y \in \R^d$, $t \in [0,1]$ such that $z:= y+tu_{(i)} \in (O \times (0,1)^{d-1}) \cup ([0,1]^{d})^c$, we have that $g \equiv \mathrm{Id}$ a neighbourhood of $z$ and $\nabla_g f_i(z) = u_{(i)}$. 
Then by definition we have that $G_i(z) |\dot{\gamma}_t^{y,i}|=G_i|u_{(i)}|_{euc} = \langle \nabla_{euc} f_i, u_{(i)} \rangle_{euc} = \partial_{u_{(i)}}f_i.$
Similarly, we have that for all points $y \in \R^d$, $t \in [0,1]$ such that $z:= y+tu_{(i)} \in (N_2 \cup N_g)^c \setminus ( O \times (0,1)^{d-1} \cup ([0,1]^d)^c)$ that $\partial_{u_{(i)}}f_i(z) = G_i(z)|\dot{\gamma}_t^{y,i}|$ (see \eqref{derivative_function_along_curve} and the equality cases of \eqref{case_1f_1}, \eqref{case_2f_2}, \eqref{case_2f_3}, and \eqref{case_2f_4}).

Hence, we get that for every $y$ such that $\gamma^{y,i}$ is transversal to $N_2 \cup N_g$, it holds
\begin{align}\label{difference_along_curve}
     f_i(\gamma^{y,i}_1)-  f_i(\gamma^{y,i}_0) = \int_0^1  \partial_{u_{(i)}}f_i(\gamma^{y,i}_t) \,\di t = \int_0^1 G_i(z)|\dot{\gamma}_t^{y,i}|\,\di t.
\end{align}
On the other hand, Lemma \ref{testplans_on_transversal} together with \eqref{evaluation_test_plan_comparison} and \eqref{difference_along_curve} yields that 
\begin{align*}
 &\int_0^1 \int_{C^0([0,1], \R^d)} |Df_i|_w(\gamma_t)|\dot{\gamma_t}|\,\di \boldsymbol{\pi}(\gamma) \di t \\
    &\leq \int_0^1 \int_{C^0([0,1], \R^d)} G_i(\gamma_t)|\dot{\gamma_t}|\,\di \boldsymbol{\pi}(\gamma) \di t-\frac{\tau\e}{\sqrt{2}} \\
    &= \int_0^1 \int_{C^0([0,1], \R^d)\cup \{\gamma \perp N_2\cup N_g\}} G_i(\gamma_t)|\dot{\gamma_t}|\,\di \boldsymbol{\pi}(\gamma) \di t-\frac{\tau\e}{\sqrt{2}}\\
    &< \int_{C^0([0,1], \R^d)\cup \{\gamma \perp N_2\cup N_g\}}| f_i(\gamma_1^{y,i})-f_i(\gamma_0^{y,i})|\, \di \boldsymbol{\pi}(\gamma) = \int_{C^0([0,1], \R^d)}| f_i(\gamma_1^{y,i})-f_i(\gamma_0^{y,i})|\, \di \boldsymbol{\pi}(\gamma).
\end{align*}
This contradicts the definition of a test plan and hence proves the Lemma. 
\end{proof}
Now we have got all the ingredients to prove the following: 
\smallskip

\begin{proposition}\label{first_non_infinitesimally_hilb}
    The metric measure space $(\R^d, \sfd_g, \vol_g)$ is not infinitesimally Hilbertian. 
\end{proposition}
\begin{proof}
    Pick a function $\phi \in C_c^\infty(\R^d)$ such that $\phi \equiv 1$ on $[-1, 2]^d$. Consider the functions $\Tilde{f}_i := \phi f_i$ for $1 \leq i \leq 4$. Then by the previous observations, we have that 
    \begin{align*}
        &|D\Tilde{f}_i|_w = G_i \ \mathrm{in}\ (-1, 2)^d\\
        &|D\Tilde{f}_i|_w = \frac{1}{\sqrt{2}}|\nabla_{euc} \Tilde{f}_i|_{euc} \ \mathrm{in}\ ((-1,2)^d)^c.
    \end{align*}
    More precisely, this gives
    \begin{align}
        |D\Tilde{f}_i|_w= |\nabla_g \Tilde{f}_i|_g \ \mathrm{a.e.\ in} \ (O\times(0,1)^{d-1}) \cup ((0,1)^{d})^c,
    \end{align}
    and
    \begin{align*}
        &|D\Tilde{f}_1|_w= \frac{1}{\sqrt{2}}, \ \mathrm{a.e.\ in} \ (0,1)^{d}\setminus O\times(0,1)^{d-1}, \\
        &|D\Tilde{f}_2|_w=|D\Tilde{f}_3|_w=|D\Tilde{f}_4|_w=1 \ \mathrm{a.e.\ in} \ (0,1)^{d}\setminus (O\times(0,1)^{d-1}).
    \end{align*}
    Then it immediately follows that 
    \begin{align}
    |D\Tilde{f}_3|_w^2 + |D\Tilde{f}_4|_w^2 = 2(|D\Tilde{f}_1|_w^2 + |D\Tilde{f}_2|_w^2)   \ \mathrm{a.e.\ in} \ (O\times(0,1)^{d-1}) \cup ((0,1)^{d})^c,
    \end{align}
    but 
    \begin{align}
        |D\Tilde{f}_3|_w^2 + |D\Tilde{f}_4|_w^2 - 2(|D\Tilde{f}_1|_w^2 + |D\Tilde{f}_2|_w^2) = -1 \ \mathrm{a.e.\ in} \ (0,1)^{d}\setminus (O\times(0,1)^{d-1}).
    \end{align}
    Moreover, we have that $\Tilde{f}_3 = \Tilde{f}_1 + \Tilde{f}_2$ and $\Tilde{f}_4 = \Tilde{f}_1 - \Tilde{f}_2$ and all four functions have finite Cheeger energy.
    But as $\mathcal{L}^d((0,1)^{d}\setminus O\times(0,1)^{d-1})\geq 1-2\kappa >0$ and the fact that $\vol_g$ and $\mathcal{L}^d$ are $2^{\frac{d}{2}}$-equivalent, we get that 
    \begin{align}
        \Ch(\Tilde{f}_3) + \Ch(\Tilde{f}_4) \neq 2(\Ch(\Tilde{f}_1) + \Ch(\Tilde{f}_2)),
    \end{align}
    which proves that the metric measure space $(\R^d, \sfd_g,  \vol_g)$ is not infinitesimally Hilbertian.
\end{proof}

\subsection{An example of a Sobolev Riemannian metric, neither quasi-Riemannian nor infinitesimally Hilbertian}\label{subsec:sobolev}
 Fix a $d \geq 3$ and set $p := d-1$. We will now refine the previous examples to construct a Riemannian metric $g$ on $\R^d$ such that $g, g^{-1} \in L^{\infty}$, $g \in W^{1,p}_{\rm loc}$ such that $(M, \sfd_g, g)$ is not quasi-Riemannian and such that the metric measure space $(M, \sfd_g, 
 \vol_g)$ is not infinitesimally Hilbertian.  

 Define $\omega_{d-1}:= \mathcal{L}^{d-1}(B^{\R^{d-1}}_1(0))$. We start with the following preparatory result:
 \smallskip
 
 \begin{lemma}\label{good_sobolev_function}
     For each $R>0$ and $\e>0$, there exists an $r >0$ and a function $f \in C^{0,1}_c(B^{\R^{d-1}}_R(0), [0,1])$ with $f(B_r(0)) = \{1\}$ and such that
     \begin{align}
         \int_{\R^{d-1}} |\nabla_{euc}f|^{d-1}\, \di x < \e.
     \end{align}
 \end{lemma}
 \begin{proof}
     This proof uses the idea of the proof of \cite[Theorem 4.16]{evans2025measure}. For each $\rho>0$ consider the function $\eta_\rho \in C_c^{0,1}(B^{\R^{d-1}}_\rho(0), [0,1])$ defined by
     \begin{align}
           \eta_\rho(y) := \left\{ \begin{array}{ll}
           1 & \mathrm{if}\ y \in B_{\rho/4}(0), \\
           1-\frac{4(|y|-\frac{\rho}{4})}{\rho}\  & \mathrm{if} \ y \in B_{\rho/2}(0) \setminus B_{\rho/4}(0) \\
           0 & \mathrm{if}\ y \in B^c_{\rho/2}(0).
        \end{array}\right.
     \end{align}
     We can estimate that
     \begin{align}
         \int_{\R^{d-1}} |\nabla_{euc} \eta_\rho|^{d-1} \, \di x \leq  \int_{B_{\rho/2}(0)} \frac{4^{d-1}}{\rho^{d-1}} \leq 2^{d-1}\omega_{d-1}.
     \end{align}
     Define 
     \begin{align*}
         S_j:= \sum_{k=1}^j \frac{1}{k}.
     \end{align*}
     Now for each $j \in \N$ define the function $f_j \in C^{0,1}_c(B_R(0), [0,1])$ as
     \begin{align}
         f_j:= \frac{1}{S_j} \sum_{k=1}^j\frac{\eta_{R/8^k}}{k}.
     \end{align}
     Then the construction of $\eta_\rho$ implies that on $B_{R/8^{j+1}}(0)$ we have that $f_j \equiv 1$. Moreover, we have that $\supp \, \nabla_{euc}\eta_{R/8^k} = \overline{B_{R/(2 \cdot 8^k)}(0)} \setminus B_{R/(4 \cdot 8^k)}(0)$, hence for any $k \neq k' \in \N$, we get that $\supp \, \nabla_{euc}\eta_{R/8^k} \cap \supp \, \nabla_{euc}\eta_{R/8^{k'}} = \emptyset$. We may now estimate
     \begin{align*}
         \int_{\R^{d-1}} |\nabla_{euc}f_j|^{d-1}\, \di x = \frac{1}{S_j^{d-1}} \sum_{k=1}^j \frac{1}{k^{d-1}}\int_{\R^{d-1}} |\nabla_{euc} \eta_{R/8^k}|^{d-1}\, \di x \leq \frac{2^{d-1}\omega_{d-1}}{S_j^{d-1}} \sum_{k=1}^j \frac{1}{k^{d-1}}.
     \end{align*}
     Now as $d-1>1$, we get that $\lim_{j \to \infty} \frac{2^{d-1}\omega_{d-1}}{S_j^{d-1}} \sum_{k=1}^j \frac{1}{k^{d-1}} =0$. Hence, we can find a $j \in \N$ such that $\int_{\R^{d-1}} |Df_j|^{d-1}\, \di x < \e$. Then setting $r:= \frac{R}{8^{j+1}}$ finishes the proof. 
 \end{proof}

 Pick a decreasing sequence $(R_m)_{m \in \N} \subset (0,1)$ such that $R_0 = \frac{1}{2}$ and for $m \geq 1$ it holds
\begin{align}
    & R_m \leq \frac{1}{2^{m+3}}, \label{distance_control_between_binary_gridpoints}\\
    &\sum_{m \in \N} 2^{(m+1)d+4}R_m < \frac{1}{4(1+\omega_{d-1})} \label{lebesgue_and_diameter_control}.
\end{align}
We inductively define the finite sets $D_m \subset (0,1)^{d-1}$, open sets $O_m \subset (-1,2)^{d-1}$, a decreasing null sequence $(r_m)_{m \in \N} \subset (0,1)$ and functions $\theta_m \in C^\infty(\R^{d-1})$ as follows: 
\begin{align*}
    D_0:=O_0 := \emptyset, \ \theta_0 \equiv 2 \in C^\infty(\R^{d-1}), r_0 := \frac{1}{4}.
\end{align*}
Now for $m \geq 1$, define 
\begin{align}
    &D_m := \Big(\frac{1}{2^{m+1}} \Z^{d-1} \cap (0,1)^{d-1}\Big) \setminus \overline{\big(\bigcup_{j \leq m-1}O_j\big)}, \label{set_D_m} \\
    &r_m:= \min\Big(R_m, \frac{1}{2}\dist_{euc}(D_m, \bigcup_{j \leq m-1}O_j), \frac{r_{m-1}}{2}\Big), \label{set_r_m} \\
    &O_m := \bigcup_{x \in D_m} B_{r_m}(x) \label{set_O_m}. 
\end{align}
Note that \eqref{set_r_m} together with \eqref{set_O_m} yield that 
\begin{align}\label{disjoint_Os}
    \overline{O_m} \cap \overline{O_j} = \emptyset \ \mathrm{for\ all\ } j < m.
\end{align}
Moreover, note that from $D_m \subset \frac{1}{2^{m+1}} \Z^{d-1}$, together with \eqref{distance_control_between_binary_gridpoints} and \eqref{set_r_m}, we have that for $x, x' \in D_m$, $x \neq x'$, it holds 
\begin{align}\label{m-th_gridpoints_are_separated}
    \overline{B_{r_m}(x)} \cap \overline{B_{r_m}(x')} = \emptyset.
\end{align}
Using Lemma \ref{good_sobolev_function}, we can for each $m \in \N$ find a function $\eta_m \in C_c^{0,1} (\R^{d-1}, [0,1])$ such that there exists an $\rho_m \in (0, r_m/5)$ with $\eta_m|_{B_{\rho_m}(0)} \equiv 1$, $\eta_m|_{\R^{d-1} \setminus B_{r_m/5}(0)} \equiv 2$ and 
\begin{align}\label{eta_m_sobolev_norm}
    2^{(m+1)(d+1)}\int_{\R^{d-1}} |D\eta_m|^{d-1}\, \di x \leq 2^{-m}.
\end{align}
Define the function $\psi_m \in C^\infty(\R^{d-1})$ via 
\begin{align}\label{def_psi_m}
    \psi_m(y) := \left\{ \begin{array}{ll}
           2 & \mathrm{if}\ y \in O_m^c, \\
           \eta_m\big(y-x\big)\  & \mathrm{if} \ y \in B_{r_m}(x) \ \mathrm{for\ some\ } x \in D_m.
        \end{array}\right.
\end{align}
By \eqref{set_O_m} and \eqref{m-th_gridpoints_are_separated}, we get that $\psi_m$ is well-defined. 
Now define
\begin{align}\label{set_theta_m}
    \theta_m := \min(\theta_{m-1}, \psi_m). 
\end{align}
We observe that by \eqref{lebesgue_and_diameter_control}, it holds
\begin{align}\label{union_of_binary_balls_small_measure}
    \mathcal{L}^{d-1}\big(\bigcup_{m \in \N}\overline{O_m}\big) \leq \sum_{m=1}^\infty \#(\frac{1}{2^{m+1}}\Z^{d-1}\cap (0,1)^{d-1})\mathcal{L}^{d-1}(B_{r_m}(0)) \leq  \omega_{d-1}\sum_{m=1}^\infty2^{(m+1)(d-1)+1} R_m^{d-1} \leq \frac{1}{4}. 
\end{align}
Define the set
\begin{align}\label{def_set_s}
    S:= (0,1)^{d-1} \setminus \bigcup_{m \in \N} \overline{O_m}.
\end{align}

\begin{lemma}\label{D_m_dense}
$\bigcup_{m \in \N} D_m$ is dense in $S$.
\end{lemma}
\begin{proof}
     Assume this was not the case. Then there exists a $y \in S$ and an $r >0$ such that $B_r(y) \cap \bigcup_{m \in \N} D_m = \emptyset$, where $B_r(y)=B^{\R^{d-1}}_{r}(y)$. 
By the monotonicity of the sequence $R_m$, there exists a $j \in \N$ such that 
\begin{align}\label{pick_j_large_for_small_R}
    R_{j'} \leq \frac{r}{3}, \ \mathrm{for}\ j' \geq j.
\end{align}
Moreover, as by definition $\bigcup_{k \leq j} \overline{O_k} \not \ni y$ is compact, there exists an $r' \in (0, \frac{r}{3})$ such that 
\begin{align}\label{r_prime_stay_away_from_first_j}
    B_{r'}(y) \cap \bigcup_{k \leq j} \overline{O_k} = \emptyset.
\end{align}
Now as $(\bigcup_{l \in \N} \frac{1}{2^{l+1}} \Z^{d-1}\setminus \bigcup_{l \leq j} \frac{1}{2^{l+1}} \Z^{d-1})\cap (0,1)^{d-1}$ is dense in $(0,1)^{d-1}$, can find $j<l \in \N$, $z \in \frac{1}{2^{l+1}}\Z^{d-1} \cap (0,1)^{d-1}$, such that $z \in B_{r'}(y)$. 
Then by our assumption, we have that $z \notin \bigcup_{m \in \N} D_m$, hence $z \notin D_l$. 
Considering the $l$-th step of the above construction and \eqref{set_D_m}, we get that $z \notin D_l$ implies
\begin{align*}
    z \in \bigcup_{k < l} \overline{O_k}. 
\end{align*}
But then \eqref{r_prime_stay_away_from_first_j} implies that 
\begin{align*}
    z \in \bigcup_{j < k < l} \overline{O_k}. 
\end{align*}
Hence there exists a $j<k <l$ such that $z \in \overline{O_k}$, which together with \eqref{set_O_m} implies that $z \in B_{r_k}(x)$ for some $x \in D_k$. But then \eqref{pick_j_large_for_small_R} and \eqref{r_prime_stay_away_from_first_j} give that 
\begin{align*}
    |y-x|_{euc} \leq |z-y|_{euc} + |z-x|_{euc} \leq r_k+ r' \leq R_j + r' < \frac{2r}{3}.
\end{align*}
But then $x \in B_r(y) \cap D_k$, which gives a contradiction. 
\end{proof}

Now, we will prove the following useful property:
\smallskip

\begin{lemma}\label{shrinking_ball_om_behaviour}
    For each point $y \in \R^{d-1} \setminus \bigcup_{m\in \N} \overline{O_m}$, we have that 
    \begin{align}
        &\lim_{r \to 0} \frac{\sum_{m \in \N}r_m\#\{x \in D_m: \overline{B_{r_m/5}(x)} \cap \overline{B_r(y)}\neq \emptyset\}}{r} =0 \ \mathrm{and} \label{diameter_sum_shrinking_balls} \\
        &\lim_{r \to 0} \frac{\sum_{m \in \N}r^{d-1}_m\#\{x \in D_m: \overline{B_{r_m/5}(x)} \cap \overline{B_r(y)}\neq \emptyset\}}{r^{d-1}} =0. \label{volume_sum_shrinking_ball}
    \end{align}
\end{lemma}
\begin{proof}
    First, denote 
    \begin{align}
        &\zeta_{m}(y, r):= \#\{x \in D_m: \overline{B_{r_m/5}(x)} \cap \overline{B_r(y)}\neq \emptyset\}.
    \end{align}
    Fix an $m \in \N$.
    As $y \notin \overline{O_m}$, we get that for each $x \in D_m$, it holds $|x-y|_{euc} > r_m$, hence 
    \begin{align}
        \dist_{euc}(y, \overline{B_{r_m/5}(x)}) > \frac{4r_m}{5}.
    \end{align}
Thus, if $r \leq \frac{4r_m}{5}$, it follows that $\zeta_m(y, r) =0$. It now suffices to consider those $m$ such that $r_m \leq \frac{5}{4}r$. In that case, we get 
    \begin{align*}
        \zeta_m(y,r) &\leq \#\big\{ \frac{1}{2^{m+1}}\Z^{d-1} \cap \overline{B_{r+\frac{r_m}{5}}(y)} \neq \emptyset\big\} \leq C(d) 2^{(d-1)(m+1)} \frac{5^{d-1}}{4^{d-1}} r^{d-1},
    \end{align*}
    where $C(d)$ is a constant only depending on $d$. Then it follows that 
    \begin{align}
        \frac{\sum_{m \in \N}r^{d-1}_m\#\{x \in D_m: \overline{B_{r_m/5}(x)} \cap \overline{B_r(y)}\neq \emptyset\}}{r^{d-1}} &\leq C(d)\sum_{m:\, 4r_m \leq 5r} 2^{(d-1)(m+1)}r_m^{d-1} \nonumber\\
        &\leq C(d)\sum_{m:\, 4r_m \leq 5r} 2^{(d-1)(m+1)}r_m. \label{estimate_volume_sum}
    \end{align}
    Now \eqref{lebesgue_and_diameter_control} implies that the right-hand-side of \eqref{estimate_volume_sum} converges to $0$ as $r \to 0$, which implies \eqref{volume_sum_shrinking_ball}.
    Moreover, 
    \begin{align}
        \frac{\sum_{m \in \N}r_m\#\{x \in D_m: \overline{B_{r_m/5}(x)} \cap \overline{B_r(y)}\neq \emptyset\}}{r} &\leq r^{d-2}C(d)\sum_{m:\, 4r_m \leq 5r} 2^{(d-1)(m+1)}r_m, 
    \end{align}
    which using \eqref{lebesgue_and_diameter_control} once again yields \eqref{diameter_sum_shrinking_balls}.
\end{proof}
\begin{proposition}
    The functions $\theta_m$ converge to a function $\theta$ in $W^{1,d-1}_{\rm loc}(\R^{d-1})$ as $m \to \infty$.
\end{proposition}
\begin{proof}
    We first note that for all $m$, it holds $\theta_m \equiv 2$ outside $(-1, 2)^{d-1}$, hence, it suffices to consider its behaviour inside $(-1, 2)^{d-1}$. Note that \eqref{def_psi_m} implies that $\psi_m^{-1}([1, 2)) \subset O_m$. Then \eqref{disjoint_Os} together with \eqref{set_theta_m} yields that for any $m' < m$ it holds 
    \begin{align}\label{explicit_theta_for_cauchy_sequence}
        \theta_m = \theta_{m'} + \sum_{j=m'+1}^{m} (\psi_j-2).
    \end{align}
    Then, as $(\psi_j-2) \in C_c^\infty((-1,2)^{d-1})$, the Poincaré inequality together with \eqref{m-th_gridpoints_are_separated} and \ref{eta_m_sobolev_norm} yields that 
    \begin{align}
        \norm{(\psi_j-2)}^{d-1}_{W^{1,d-1}(\R^{d-1})} &\leq C \int_{(-1, 2)^{d-1}}|\nabla_{euc}\psi_j|^{d-1}\, \di \mathcal{L}^{d-1} \nonumber \\
        &= C|D_j| \int_{B_{r_j}(0)} |\nabla_{euc}\eta_j|^{d-1}(x) \, \di \mathcal{L}^{d-1}(x) \nonumber \\
        & \leq C2^{(j+1)(d-1)}2^{-j(d+1)} \leq C2^{-j} \label{sobolev_estimate_psi_j}.
    \end{align}
    Hence, for $m' \leq m$, we have that 
    \begin{align*}
        \norm{\nabla_{euc}\theta_m -\nabla_{euc}\theta_{m'}}_{L^{d-1}((-1, 2)^{d-1})}^{d-1} \leq C\sum_{j=m'+1}^{m} 2^{-j}.
    \end{align*}
    Note that the right hand side goes to $0$ as $m' \to \infty$, thus, $(\nabla_{euc}\theta_m)_m$ is a null-sequence in $L^{d-1}(\R^{d-1})$. As for all $m$ one has that $\theta_m-2\in C_c^\infty((-1, 2)^{d-1}, \R)$, the Poincaré inequality yields that $\theta_m$ converges to a function $\theta$ in $W^{1,d-1}_{\rm loc}(\R^{d-1})$. 
\end{proof}
Our inductive construction also implies that $\theta_m$ converges pointwise and is bounded in $L^\infty$, hence, we can describe $\theta$ explicitly:
For all $y \in \R^{d-1}\setminus \bigcup_{m \in \N} \overline{O_m}$, we have that $\theta_m(y)=2$ for all $m \in \N$. For all $y \in \bigcup_{m \in \N} \overline{O_m}$, there exists a $m^*$ such that $y \in \overline{O_{m^*}}$. Then for any $m \geq m^*$, \eqref{disjoint_Os} guarantees that $\theta_m(y) = \theta_{m^*}(y)$. Now as by the dominated convergence theorem, $\theta$ is the pointwise limit of the $\theta_m$, we can conclude that
\begin{align}
    &1 \leq \theta \leq 2 \ \mathrm{on} \ \R^{d-1},\\
    &\theta \equiv 2 \ \mathrm{on} \ \R^{d-1}\setminus \bigcup_{m \in \N} \overline{O_m}, \\
    & \theta^{-1}([1,2)) \subset \bigcup_{m \in \N} \bigcup_{x \in D_m} B_{r_m/5}(x).
\end{align}
Moreover, we get that for each $m \in \N$, $x\in D_m$, it holds 
\begin{align}
    \theta \equiv 1 \ \mathrm{on}\ B_{\rho_m/10}(x).
\end{align}
\begin{proposition}\label{lebesgue_pts_theta}
    The set $N_\theta \subset \R^{d-1}$ of non-Lebesgue points of $\theta$ can be chosen as a subset of the $\mathcal{L}^{d-1}$-null set $\bigcup_{m \in \N} \partial O_m$.
\end{proposition}
\begin{proof}
    For each point $y\in \R^{d-1} \setminus \bigcup_{m \in \N} \overline{O_m}$, we can apply \eqref{volume_sum_shrinking_ball} and get that 
    \begin{align*}
        \frac{1}{\mathcal{L}^{d-1}(B_r(y))}\int_{B_r(y)} |\theta(x)-2| \di \mathcal{L}^{d-1} &\leq  \frac{\mathcal{L}^{d-1}(B_r(y) \cap \bigcup_{m\in \N} \bigcup_{x \in D_m} B_{r_m/5}(x))}{\mathcal{L}^{d-1}(B_r(y))} \\
        & \leq C \frac{\sum_{m \in \N}r^{d-1}_m\#\{x \in D_m: \overline{B_{r_m/5}(x)} \cap \overline{B_r(y)}\neq \emptyset\}}{r^{d-1}} \to 0\ \mathrm{as}\ r \to 0.
    \end{align*}
    For each point $y \in \bigcup_{m \in \N} O_m$, we have that $\theta$ is continuous in a neighbourhood of $y$, hence $y$ is a Lebesgue point. For each $m$, we have that $\partial O_m$ is a $\mathcal{L}^{d-1}$-null set, hence so is their countable union. 
\end{proof}

We are now ready to define the metric $g$ on $M:= \R^{d}$ via
\begin{align}
    g(x', x_d) = \theta(x') \mathrm{Id},\ x' \in \R^{d-1}, x_d \in \R. 
\end{align}

Again $(M, g)$ satisfies Assumption \ref{assumption_on_mfd} and $(M, \sfd_g, \vol_g)$ satisfies \eqref{condition_on_measure_bounded_compression}.  

\subsubsection{Analysis of quasi-Riemannianity}\label{subsub_quasi_riem}
We now have enough information to prove Theorem \ref{snd_counterexample}.
Note that by \eqref{union_of_binary_balls_small_measure} and \eqref{def_set_s}, we have that $S \subset (0,1)^{d-1}$ satisfies $\mathcal{L}^{d-1}(S)>0$ and $\theta = 2$ on $S$ (see Proposition \ref{lebesgue_pts_theta}).
Then consider $E := S \times (0,\frac{1}{2}) \subset M$. Then $\vol_g(E)\geq \mathcal{L}^d(E) >0$. For $y \in E$, define the curve $\gamma^y:[0,1] \to (0,1)^{d-1} \times (-2, 2), t \mapsto y+ te_d$ and define the test plan $\boldsymbol{\pi} \in \mathcal{P}(C([0,1], \R^d))$ via
\begin{align*}
    \di\boldsymbol{\pi}(\gamma):= \left\{ \begin{array}{ll}
           \frac{1}{\mathcal{L}^d(E)}\di\mathcal{L}^d(y) & \mathrm{if}\ \gamma = \gamma^{y}, \ \mathrm{for \ some} \ y \in  E, \\
           0\  & \mathrm{otherwise}.
        \end{array}\right.
\end{align*}
Now the strategies from Subsection \ref{subsec:first_counter_quasi_riem} together with Lemma \ref{D_m_dense} and Proposition \ref{lebesgue_pts_theta} imply that for all $y \in E$, and almost every $t \in [0,1]$, it holds $|\dot{\gamma}^y_t|=1 < \sqrt{2} = |\dot{\gamma}^y_t|_g$. Hence, $\boldsymbol{\pi}$ is supported on curves for which the metric speed is strictly smaller than the norm of the derivative, for almost every time. 

Now we can consider the function $f\in C^1(M)$ given by $f:(x', x_d) \to x_d$. Then $f$ has slope $|Df|=1$ on $S \times \R$, but $|\nabla_g f|_g <1$ on $S \times \R$. 
Finally, $f$ and $\boldsymbol{\pi}$ together yield that $|Df|_w > |\nabla_g f|_g$ on a set of positive measure, which proves Theorem \ref{snd_counterexample}.

We will conclude this observation with the following remark.

\smallskip
\begin{remark}
    From the Sobolev embedding theorem, we get that for a $d$-dimensional manifold with a Riemannian metric $g \in W^{1,p}_{\rm loc}$ for some $p > d$, that $g$ is automatically continuous. Then the results we recalled in Subsection \ref{subsec:2.2}, yield that $(M, \sfd_g, g)$ satisfies condition (iii) from Definition \ref{def:quasi_riemannian_manifold} for $C^1$-functions. Theorem \ref{snd_counterexample} shows that this is no longer given if $d \geq 3$ and $p \leq d-1$ and therewith works towards solving Problem 2 in \cite{de1990conversazioni}, Chapter 3: spazi metrici quasi-riemanniani: the question what regularity assumption on $g$ is necessary such that the $\sfd_g$-slope of a Lipschitz function equals the $g$-norm of the gradient almost everywhere  (see also (2.2), Theorem (2.18), and Theorem (6.1) in \cite{de1991integral}). 
\end{remark}

\subsubsection{Analysis of infinitesimal Hilbertianity}\label{subsub:hilbertian}
From here, we will analyse the infinitesimal structure of the metric measure space $(M, \sfd_g, \vol_g)$ and prove Theorem \ref{hilbertianity_counterexample}.
Pick a point $y=(y', y_d) \in M$ such that $y' \in S$. We will work in a neighbourhood around $y$ and restrict to local coordinates. Pick a vector $u \in \R^d$. 
We want to understand the metric speed at time $0$ of the curve $\gamma^{y, u}: [0, \xi] \to \R^d, t \mapsto y+tu$, $\xi >0$. 
Define the null set 
\begin{align}
    N_g:= (\partial [0,1]^{d-1} \cup \bigcup_{m \in \N} \partial O_m) \times \R \subset M. 
\end{align}
Fix an $\e>0$ and find an $r_\e>0$ such that for $r \in (0,r_\e)$, the expressions in \eqref{diameter_sum_shrinking_balls} and \eqref{volume_sum_shrinking_ball} are smaller than or equal to $\e$. 
Now for some $0 <r < r_\e$ choose $z \in M$ such that $\sfd_{euc}(z,y) = \frac{r}{12}$, which using that $\sfd_g$ and $\sfd_{euc}$ are $\sqrt{2}$-equivalent implies that any $\frac{4}{3}$-shortest curve transversal to $N_g$ (i.e. a curve $\gamma \perp N_g$ from $y$ to $z$, such that $L_g(\gamma) \leq \frac{4}{3}\sfd_g(y,z)$) from $y$ to $z$ lies inside $B_r^{euc}(y)$. Pick such a curve $\gamma$ and assume it to be $1$-$\sfd_{euc}$-Lipschitz (which can be done if $r$ is small enough, as $\sfd_g$ and $\sfd_{euc}$ are $\sqrt{2}$-equivalent). 

Denote 
\begin{align}
    \Omega_m := \bigcup_{x \in D_m} B^{\R^{d-1}}_{r_m/5}(x) \subset (-1, 2)^{d-1}.
\end{align}
%
We first note that by the definition of the metric $g$ it holds that 
\begin{align}\label{length_of_curve_splitted_in_sets}
    L_g(\gamma) \geq  \int_{\gamma^{-1}((\bigcup_{m\in \N}\overline{\Omega_m})^c\times \R)} \sqrt{2}|\dot{\gamma}_t|_{euc} \, \di t + \int_{\gamma^{-1}((\bigcup_{m\in \N}\overline{\Omega_m})\times \R)} |\dot{\gamma}_t|_{euc} \, \di t.
\end{align}
This is in the spirit of the step 2 in Subsection \ref{subsec:first_counter_inf_hilbertian}, with the sole difference that we do not explicitly know $\gamma^{-1}(\{g \equiv \mathrm{Id}\})$, but we can use that $\gamma^{-1}(\{g \equiv \mathrm{Id}\}) \subset \gamma^{-1}((\bigcup_{m\in \N}\overline{\Omega_m})\times \R)$. Very roughly speaking, we will show that an efficient curve $\gamma$ mainly moves in the $d$-th coordinate direction in $\gamma^{-1}((\bigcup_{m\in \N}\overline{\Omega_m})\times \R)$. 

Now observe that 
\begin{align*}
    \gamma^{-1}\Big(\bigcup_{m \in \N}\overline{\Omega_m}\times \R\Big) = \bigcup_{m \in \N} \gamma^{-1}(\overline{\Omega_m}\times \R). 
\end{align*}
Hence, we can pick a $J \in \N$ such that 
\begin{align}\label{small_measure_preimage_choice_n}
    \mathcal{L}^1\Big(\gamma^{-1}\Big(\bigcup_{m \in \N}\overline{\Omega_m}\times \R\Big)\Big) \leq \frac{r\e}{3} + \mathcal{L}^1\Big(\gamma^{-1}\Big(\bigcup_{m \leq J}\overline{\Omega_m}\times \R\Big)\Big). 
\end{align}
As $\gamma$ is $1$-$\sfd_{euc}$-Lipschitz continuous, \eqref{small_measure_preimage_choice_n} implies that $\gamma$ is very ``short" in $\gamma^{-1}(\bigcup_{m > J}\overline{\Omega_m}\times \R)$, so, again very roughly speaking, it suffices to study $\gamma$ on $\gamma^{-1}(\bigcup_{m \leq J}\overline{\Omega_m}\times \R)$ to understand $\gamma$ on $\gamma^{-1}(\bigcup_{m \in \N}\overline{\Omega_m}\times \R)$. 

For any $x \in \bigcup_{m \in \N} D_m$, write $m(x)$ as the natural number such that $x \in D_{m(x)}$. Now the set  $\bigcup_{m \leq J} D_m$ is finite so we may write it as $\{x_l\}_{l=1}^k \subset \R^{d-1}$ for some $k \in \N$. Write $m_l:= m(x_l)$. In the following, we will essentially argue that the most efficient way to get from $\gamma_0$ to $\gamma_1$ is to pass the set $\overline{B^{\R^{d-1}}_{r_{m_l}/5}(x_l)} \times \R$ at most once for each $l \in \{1, \ldots, k\}$.

We will inductively define Lipschitz-continuous curves $\gamma^{(l)}:[0,1] \to \R^d$ from $y$ to $z$ for $l=0, \ldots, k$ as follows:
Define $\gamma^{(0)}:= \gamma$. 
Now for $l \geq 1$ define 
\begin{align*}
    &t_l^{\min}:= \min(t \in [0,1]: \gamma^{(l-1)}_t \in \overline{B^{\R^{d-1}}_{r_{m_l}/5}(x_l)}\times \R), \\
    &t_l^{\max}:= \max(t \in [0,1]: \gamma^{(l-1)}_t \in \overline{B^{\R^{d-1}}_{r_{m_l}/5}(x_l)}\times \R). 
\end{align*}
Define
\begin{align}
    w_l:= \int_{t_l^{\min}}^{t_l^{\max}} \dot{\gamma}^{(l-1)}_t \, \di t = \gamma_{t_l^{\max}}- \gamma_{t_l^{\min}} \in \R^d, 
\end{align}
and \begin{align*}
     \gamma^{(l)}_t:= \left\{ \begin{array}{ll}
           \gamma^{(l-1)}_t & \mathrm{if}\ t \in [0, t_l^{\min}], \\
           \gamma^{(l-1)}_{t_l^{\min}} + \frac{t-t_l^{\min}}{t_l^{\max}-t_l^{\min}}w_l & \mathrm{if} \ t \in [t_l^{\min}, t_l^{\max}],\\
           \gamma^{(l-1)}_t & \mathrm{if}\ t \in [t_l^{\max}, 1].
        \end{array}\right.
\end{align*}
\textbf{Claim 1.} For each $l = 1, \ldots, k$, $\gamma^{(l)}$ satisfies 
\begin{align}\label{l-th_curve_preimage_length_estimate}
    &\int_{(\gamma^{(l-1)})^{-1}((\bigcup_{m\in \N}\overline{\Omega_m})^c\times \R)} \sqrt{2}|\dot{\gamma}^{(l-1)}_t|_{euc} \, \di t + \int_{(\gamma^{(l-1)})^{-1}((\bigcup_{m\in \N}\overline{\Omega_m})\times \R)} |\dot{\gamma}^{(l-1)}_t|_{euc} \, \di t \nonumber \\
    &\geq \int_{(\gamma^{(l)})^{-1}((\bigcup_{m\in \N}\overline{\Omega_m})^c\times \R)} \sqrt{2}|\dot{\gamma}^{(l)}_t|_{euc} \, \di t + \int_{(\gamma^{(l)})^{-1}((\bigcup_{m\in \N}\overline{\Omega_m})\times \R)} |\dot{\gamma}^{(l)}_t|_{euc} \, \di t.
\end{align}
\begin{proof}[Proof of claim 1]
This can be seen by noticing that 
\begin{align}\label{preimages_varying_in_lth_transformation_of_curve}
    &(\gamma^{(l)})^{-1}((\bigcup_{m\in \N}\overline{\Omega_m})^c\times \R) = (\gamma^{(l-1)})^{-1}((\bigcup_{m\in \N}\overline{\Omega_m})^c\times \R) \setminus [t_l^{\min}, t_l^{\max}] \ \mathrm{and} \nonumber \\
    &(\gamma^{(l)})^{-1}((\bigcup_{m\in \N}\overline{\Omega_m})\times \R) = (\gamma^{(l-1)})^{-1}((\bigcup_{m\in \N}\overline{\Omega_m})\times \R) \cup [t_l^{\min}, t_l^{\max}].
\end{align}
Moreover, $\gamma^{(l-1)} = \gamma^{(l)}$ on $[t_l^{\min}, t_l^{\max}]^c$ and, as $\overline{B^{\R^{d-1}}_{r_{m_l}/5}(x_l)} \times \R \subset \R^d$ is convex, it holds that $\gamma^{(l)}([t_l^{\min}, t_l^{\max}]) \subset \overline{B^{\R^{d-1}}_{r_{m_l}/5}(x_l)} \times \R$. Thus, we have that
\begin{align}\label{integral_outside_l_interval}
    &\int_{[t_l^{\min}, t_l^{\max}]^c \cap (\gamma^{(l-1)})^{-1}((\bigcup_{m\in \N}\overline{\Omega_m})^c\times \R)} \sqrt{2}|\dot{\gamma}^{(l-1)}_t|_{euc} \, \di t + \int_{[t_l^{\min}, t_l^{\max}]^c \cap(\gamma^{(l-1)})^{-1}((\bigcup_{m\in \N}\overline{\Omega_m})\times \R)} |\dot{\gamma}^{(l-1)}_t|_{euc} \, \di t \nonumber \\
    &= \int_{[t_l^{\min}, t_l^{\max}]^c \cap(\gamma^{(l)})^{-1}((\bigcup_{m\in \N}\overline{\Omega_m})^c\times \R)} \sqrt{2}|\dot{\gamma}^{(l)}_t|_{euc} \, \di t + \int_{[t_l^{\min}, t_l^{\max}]^c \cap(\gamma^{(l)})^{-1}((\bigcup_{m\in \N}\overline{\Omega_m})\times \R)} |\dot{\gamma}^{(l)}_t|_{euc} \, \di t,
\end{align}
and 
\begin{align}\label{integral_in_l_interval}
    &\int_{[t_l^{\min}, t_l^{\max}] \cap (\gamma^{(l-1)})^{-1}((\bigcup_{m\in \N}\overline{\Omega_m})^c\times \R)} \sqrt{2}|\dot{\gamma}^{(l-1)}_t|_{euc} \, \di t + \int_{[t_l^{\min}, t_l^{\max}] \cap(\gamma^{(l-1)})^{-1}((\bigcup_{m\in \N}\overline{\Omega_m})\times \R)} |\dot{\gamma}^{(l-1)}_t|_{euc} \, \di t \nonumber \\
    &\geq \int_{[t_l^{\min}, t_l^{\max}]}|\dot{\gamma}_t^{(l-1)}|_{euc} \, \di t \geq |w_l|_{euc}=  \int_{[t_l^{\min}, t_l^{\max}] \cap(\gamma^{(l)})^{-1}((\bigcup_{m\in \N}\overline{\Omega_m})\times \R)} |\dot{\gamma}^{(l)}_t|_{euc} \, \di t.
\end{align}
The sum of \eqref{integral_outside_l_interval} and \eqref{integral_in_l_interval} yields the claim.
\end{proof}

Furthermore, we observe that the observations after \eqref{preimages_varying_in_lth_transformation_of_curve} together with $\gamma^{(l)}([t_j^{\min}, t_l^{\max}]) \subset (\bigcup_{m\leq J}\overline{\Omega_m})\times \R \subset (\bigcup_{m> J}\overline{\Omega_m})^c\times \R$ imply that 
\begin{align}\label{small_measure_preimage_of_O_m_large_m}
    \mathcal{L}^1((\gamma^{(l)})^{-1}((\bigcup_{m>J }\overline{\Omega_m})\times \R)) &= \mathcal{L}^1((\gamma^{(l-1)})^{-1}((\bigcup_{m>J }\overline{\Omega_m})\times \R)\setminus [t_l^{\min}, t_l^{\max}]) \nonumber\\
    & \leq  \mathcal{L}^1((\gamma^{(l-1)})^{-1}((\bigcup_{m>J}\overline{\Omega_m})\times \R)),
\end{align}
for all $l =1, \ldots, k$.  For $l' \leq l$ denote by $I^{(l)}_{l'}:= (\gamma^{(l)})^{-1}(\overline{B^{\R^{d-1}}_{r_{m_{l'}}/5}(x_{l'})}\times \R)$.

\textbf{Claim 2.} For each $l=1, \ldots, k$, $l' \leq l$, we have that  $I^{(l)}_{l'}$ is a closed interval.
\begin{proof}[Proof of claim 2.]
We will prove this by induction. For $l=1$ the claim is true as $(\gamma^{(1)})^{-1}(\overline{B^{\R^{d-1}}_{r_{m_{1}}/5}(x_{1})}\times \R)= [t_1^{\min}, t_1^{\max}]$. Now assume the claim holds true for some $l <k$. 
Now we have that $(\gamma^{(l+1)})^{-1}(\overline{B^{\R^{d-1}}_{r_{m_{l+1}}/5}(x_{l+1})}\times \R)= [t_{l+1}^{\min}, t_{l+1}^{\max}]$. 
By the construction of $\gamma^{(l+1)}$, \eqref{disjoint_Os}, and \eqref{m-th_gridpoints_are_separated}, we have that $t_{l+1}^{\min}, t_{l+1}^{\max} \notin I^{(l)}_{l'}$.
Hence, as $I^{(l)}_{l'}$ is a closed interval by the induction hypothesis, it holds that $[t_{l+1}^{\min}, t_{l+1}^{\max}] \supset I^{(l)}_{l'}$, or  $[t_{l+1}^{\min}, t_{l+1}^{\max}] \cap I^{(l)}_{l'} = \emptyset$. 
In the first case, we get that $I^{(l+1)}_{l'} = \emptyset$ and in the second one it holds $I^{(l+1)}_{l'} = I^{(l)}_{l'} $. Both are closed intervals, which finishes the induction step and proves the claim.
\end{proof} 

Finally, we note that for all $l=0, \ldots, k$, $\gamma^{(l)}$ is $1$-$\sfd_{euc}$-Lipschitz. This is true by assumption for $l=0$. For $l >0$ this follows inductively, using that in $\R^d$, $w_l = \gamma^{(l-1)}_{t^{\max}_l}- \gamma^{(l-1)}_{t^{\min}_l}$, hence by the induction hypothesis, $|w_l|_{euc} \leq t^{\max}_l-t^{\min}_l$. 

Now we can combine \eqref{length_of_curve_splitted_in_sets} with \eqref{l-th_curve_preimage_length_estimate} to get that 
\begin{align*}
    L_g(\gamma) \geq \int_{(\gamma^{(k)})^{-1}((\bigcup_{m\in \N}\overline{\Omega_m})^c\times \R)} \sqrt{2}|\dot{\gamma}^{(k)}_t|_{euc} \, \di t + \int_{(\gamma^{(k)})^{-1}((\bigcup_{m\in \N}\overline{\Omega_m})\times \R)} |\dot{\gamma}^{(k)}_t|_{euc} \, \di t.
\end{align*}
Now \eqref{small_measure_preimage_choice_n} and \eqref{small_measure_preimage_of_O_m_large_m}, together with the fact that $\gamma^{(k)}$ is $1$-$\sfd_{euc}$-Lipschitz continuous yield that 
\begin{align}\label{gamma_length_vs_gamma_k}
    L_g(\gamma) + r\e \geq \int_{(\gamma^{(k)})^{-1}((\bigcup_{m\leq J}\overline{\Omega_m})^c\times \R)} \sqrt{2}|\dot{\gamma}^{(k)}_t|_{euc} \, \di t + \int_{(\gamma^{(k)})^{-1}((\bigcup_{m\leq J}\overline{\Omega_m})\times \R)} |\dot{\gamma}^{(k)}_t|_{euc} \, \di t.
\end{align}
As in the previous example, we define
\begin{align*}
    &v_{(1)}:= \int_{(\gamma^{(k)})^{-1}((\bigcup_{m\leq J}\overline{\Omega_m})\times \R)}\dot{\gamma}^{(k)} \, \di t \in \R^d,\\
    &v_{(2)}:= \int_{(\gamma^{(k)})^{-1}((\bigcup_{m\leq J}\overline{\Omega_m})^c\times \R)}\dot{\gamma}^{(k)} \, \di t \in \R^d.
\end{align*}
Then, 
\begin{align*}
    v_{(1)}+v_{(2)} = z-y.
\end{align*}
Applying classical facts on the Euclidean space to \eqref{gamma_length_vs_gamma_k} yields that 
\begin{align}\label{length_gamma_vs_two_straight_linea_a}
    L_g(\gamma) + r\e \geq \sqrt{2}|v_{(2)}|_{euc} + |v_{(1)}|_{euc}. 
\end{align}
Let $P^{(d-1)}: \R^d \to \R^d$ be the orthogonal projection mapping $\R^{d-1} \times \R\ni (x', x_d) \mapsto (x', 0)$. 
Now by Claim 2 together with our choice of $r_\e$ (see \eqref{diameter_sum_shrinking_balls}), we get that 
\begin{align}\label{pv_1_short_in_first_d-1_dimensions}
    |P^{(d-1)} v_{(1)}|_{euc} &= \Big|\sum_{l=1}^k \int_{(\gamma^{(k)})^{-1}(\overline{B^{\R^{d-1}}_{r_{m_l}/5}(x_l)}\times \R)} P^{(d-1)}\dot{\gamma}^{(k)}_t \, \di t \Big|_{euc} \nonumber \\
    &\leq \sum_{l=1}^k \Big|\int_{(\gamma^{(k)})^{-1}(\overline{B^{\R^{d-1}}_{r_{m_l}/5}(x_l)}\times \R)} P^{(d-1)}\dot{\gamma}^{(k)}_t \, \di t \Big|_{euc} \nonumber \\
    &= \sum_{l=1}^k |P^{(d-1)}\gamma^{(k)}(\max I^{(k)}_l)-P^{(d-1)}\gamma^{(k)}(\min I^{(k)}_l)| \nonumber \\
    &\leq \sum_{l=1}^k \diam(\overline{B^{\R^{d-1}}_{r_{m_l}/5}(x_l)}) \nonumber \\
    & \leq 2r \frac{\sum_{m \in \N}\frac{r_m}{5}|\{x \in D_m: \overline{B^{\R^{d-1}}_{r_m/5}(x)} \cap \overline{B^{\R^{d-1}}_r(y')}\neq \emptyset\}|}{r} \leq 2r\e,
\end{align}
where we wrote $y=(y', y_d) \in \R^{d-1} \times \R$. 
Now define 
\begin{align*}
    w_{(1)}:= v_{(1)}- P^{(d-1)}(v_{(1)}) =  (v_{(1)})_d e_d \in \mathrm{span}(e_d) \subset \R^d , \ w_{(2)} = v_{(2)}+  P^{(d-1)}(v_{(1)}) \in \R^d. 
\end{align*}
Then $w_{(1)}+w_{(2)}= z-y$ and 
\begin{align*}
    |w_{(1)}-v_{(1)}|_{euc}, |w_{(2)}-v_{(2)}|_{euc} \leq 2r\e. 
\end{align*}
Hence, with \eqref{length_gamma_vs_two_straight_linea_a}, we get that
\begin{align}\label{length_gamma_vs_two_straight_linea_b}
    L_g(\gamma) + 6r\e \geq \sqrt{2}|w_{(2)}|_{euc} + |w_{(1)}|_{euc}, 
\end{align}
for vectors $w_{(1)},w_{(2)}$ with $w_{(1)}+w_{(2)} = z-y$ and $w_{(1)} \in \mathrm{span}(e_d)$. We want to minimise the right hand side of \eqref{length_gamma_vs_two_straight_linea_b} under the thereafter mentioned constraints. This will be done similarly as in the previous example (see step 3, from \eqref{length_estimate_w_vectors} onwards). 

Write again $z-y = u = (u', u_d) \in \R^{d-1} \times \R$. By potentially applying a linear coordinate transform from $SO(d-1)$, we may assume that $u' \in \mathrm{span}(e_1)$. Then we reduce the right-hand-side of \eqref{length_gamma_vs_two_straight_linea_b} by orthogonally projecting $w_{(1)}$ and $w_{(2)}$ onto $\mathrm{span}(e_1, e_d)$. 
Now if $u \in \mathrm{span}(e_1)$, we get that $w_{(1)} =0$ and hence
\begin{align*}
    L_g(\gamma) + 6r\e \geq \sqrt{2}|z-y|_{euc}.
\end{align*}
Otherwise, we have that $u_d \neq 0$. Then $u$ can be written as $a(be_1 + e_d)$. Then writing $w_{(1)} = a(1-c)e_d$ and $w_{(2)} = a(be_1 + ce_d)$ reduces to minimising $f_b(c)$ as defined in \eqref{one_dim_miminisation_problem}. Together with $r = 12|u|_{euc}$, this yields 
\begin{align*}
    L_g(\gamma) \geq \left\{ \begin{array}{ll}
           |u_d| + |u'|_{euc} - 100\e|u|_{euc} & \mathrm{if}\ |u_d| \geq |u'|_{euc}, \\
           \sqrt{2}|u|_{euc}- 100\e|u|_{euc}\  & \mathrm{otherwise}.
        \end{array}\right.
\end{align*}
As before, together with $\sfd_g = \boldsymbol{\delta} \geq \boldsymbol{\delta}_{N_g}$ (see \eqref{null_set_def_of_metric}) this yields  
\begin{align}\label{asymptotic_lower_bound_metric_ex2}
    \sfd_g(y, y+u) \geq \left\{ \begin{array}{ll}
           |u_d| + |u'|_{euc} - 100\sigma(y, u)|u|_{euc} & \mathrm{if}\ |u_d| \geq |u'|_{euc}, \\
           \sqrt{2}|u|_{euc}- 100\sigma(y, u)|u|_{euc}\  & \mathrm{otherwise},
        \end{array}\right.
\end{align}
where 
\begin{align*}
    \sigma(y,u) := \frac{\sum_{m \in \N}\frac{2r_m}{5}|\{x \in D_m: \overline{B^{\R^{d-1}}_{r_m/5}(x)} \cap \overline{B^{\R^{d-1}}_{|u|_{euc}}(y')}\neq \emptyset\}|}{|u|_{euc}} \to 0,\ \mathrm{as} \ |u|_{euc}  \to 0,
\end{align*}
as of Lemma \ref{shrinking_ball_om_behaviour}.

We will next prove an upper bound on $\sfd_g(y,z)$ and work with the regularisation. Fix $\epsilon >0$. For any $\varsigma >0$ and $g_\varsigma := \rho_{\varsigma} * g$, it holds $g_\varsigma \leq 2 \mathrm{Id}$. Hence,
\begin{align*}
    \sfd_{g_\varsigma}(y,z) \leq \sqrt{2}|u|_{euc}.
\end{align*}
By Lemma \ref{D_m_dense}, we can find an $m \in \N$ and an $x \in D_m$ such that $|P^{(d-1)}(y)-(x, 0)| \leq \epsilon|u|_{euc}$. Now for $\varsigma \in (0, \rho_m/20)$, it holds $g_\varsigma =1$ on $\{x\} \times \R$. If $|u_d|> |u'|_{euc}$, write $b= \frac{|u'|_{euc}}{|u_d|}$ and note that the concatenation of the curves
\begin{align*}
    &\gamma_1:[0,1] \to \R^d, t \mapsto y+t((x,0)-P^{(d-1)}y), \\
    &\gamma_2:[0,1] \to \R^d, t \mapsto y+((x,0)-P^{(d-1)}y) + t(1-b)u_de_d, \\
    &\gamma_3:[0,1] \to \R^d, t \mapsto  y+((x,0)-P^{(d-1)}y) + (1-b)u_de_d-t((x,0)-P^{(d-1)}y),\ \mathrm{and} \\
    &\gamma_4:[0,1] \to \R^d, t \mapsto  y + (1-b)u_de_d + t((u',0)+ bu_de_d),
\end{align*}
connects $y$ with $z$ and has $g_\varsigma$ length at most $|u'|_{euc} + |u_d|_{euc} + 3 \epsilon$. As $\epsilon$ was arbitrary, we get that 
\begin{align}\label{upper_bpund_metric_ex2}
    \sfd_g(y, y+u)= \lim_{\varsigma \to 0} \sfd_{g_\varsigma}(y, y+u) \leq \alpha(u):=  \left\{ \begin{array}{ll}
           |u_d| + |u'|_{euc} & \mathrm{if}\ |u_d| \geq |u'|_{euc}, \\
           \sqrt{2}|u|_{euc}\  & \mathrm{otherwise}.
        \end{array}\right.
\end{align}
Together with Lemma \ref{metric_speed_depends_on_direction}, this gives that for any $y \in S \times \R$, the metric speed of a curve $\gamma:[0,1] \to M$, $t\in [0,1]$ such that $\dot{\gamma}_t=: u$ exists and $\gamma_t = y$ is given by $\alpha(u)$. 

Now the argument works similarly as in the previous example. 
We note that for any $y \in (([0,1]^{d-1})^c \times \R) \cup (\bigcup_{m \in \N}O_m \times \R)$ we have that $g$ is continuous in a neighbourhood of $y$ and the metric speed of a curve is given by the $g$-norm of its derivative, wherever it exists. 
\smallskip
\begin{proposition}
    The metric measure space $(\R^{d}, \sfd_g, \vol_g)$ is not infinitesimally Hilbertian. 
\end{proposition}
We will only provide a limited amount of detail, as the idea of the proof is very similar to the one of Proposition \ref{first_non_infinitesimally_hilb}.
\begin{proof}
Let $\Tilde{\phi} \in C_c^\infty(\R^{d-1})$ such that $\Tilde \phi((-2, 2)^{d-1}) = \{1\}$ and $\Tilde \phi \equiv 0$ on $((-3, 3)^{d-1})^c$. Define $\phi \in C^\infty(M)$ locally via $\phi: \R^{d-1} \times \R, (x', x_d) \mapsto \Tilde \phi (x')$ 
Now consider the functions 
\begin{align*}
    f_1: x \mapsto \phi(x) x_1, f_2: x \mapsto \phi(x) |x_d|, f_3 = f_1+f_2, f_4 = f_1-f_2. 
\end{align*}
Define the functions $G_i$ for $i =1,\ldots, 4$ $\vol_g$-almost everywhere (more precisely on $M\setminus N_g$) via 
\begin{align*}
    &G_1(z):= \left\{ \begin{array}{ll}
           \frac{1}{\sqrt{2}} & \mathrm{if}\ z \in S \times \R, \\
           |\nabla_g f_1|_{g} & \mathrm{otherwise}.
        \end{array}\right. ,G_2(z):= \left\{ \begin{array}{ll}
           1 & \mathrm{if}\ z \in S \times \R, \\
           |\nabla_g f_2|_{g}  & \mathrm{otherwise}.
        \end{array}\right. \\
        &G_3(z):= \left\{ \begin{array}{ll}
           1 & \mathrm{if}\ z \in S \times \R, \\
           |\nabla_g f_3|_{g}  & \mathrm{otherwise}.
        \end{array}\right. ,G_4(z):= \left\{ \begin{array}{ll}
           1 & \mathrm{if}\ z \in S \times \R, \\
           |\nabla_g f_4|_{g} & \mathrm{otherwise}.
        \end{array}\right. \\
\end{align*}
As in Proposition \ref{minimal_weak_upper_gradient_ex1}, one can show that $G_i$ is the minimal weak upper gradient of $f_i$ for $i=1, \ldots, 4$.
Now from the above definition, one can see that $f_i, G_i$ are in $L^\infty_{\rm loc}(M)$. From the definition, we have that outside $S \times \R$, we have that $2G_1^2+2G_2^2= G_3^2+G_4^2$ almost everywhere. Inside $S \times \R$, we have that 
\begin{align}\label{non-quadratic_cheegers}
  2G_1^2+2G_2^2-  (G_3^2+G_4^2) =1 \neq 0.  
\end{align}
As in the previous example, this is the main fact that will yield non-infinitesimal Hilbertianity. We only have to navigate around the problem that $G_i \notin L^2(\R^d,\vol_g)$. 
In a final step define a function $\eta \in C_c^\infty(\R)$ such that $\eta(x) = \eta(-x)$ for all $x \in \R$ and such that $\eta \equiv 1$ on $(-1, 1)$ and $\eta \equiv 0$ on $(-2, 2)^c$. For $n\geq 1$ define $\eta_n$ via $\eta_n \equiv 1$ on $(-n, n)$, $\eta_n(x)= \eta(|x|-n+1)$ for $|x| \geq n$ and for $1 \leq i \leq 4$, define the functions $f_{i, n}:\R^d\to \R$ via 
\begin{align*}
    &f_{1, n}: (x_1, \ldots, x_d) =x \mapsto \eta_n(x_d)\phi(x) x_1, \\
    &f_{2, n}: (x_1, \ldots, x_d) =x \mapsto \eta_n(x_d)\phi(x) (|x_d|-n+1),\\
    &f_{3, n} := f_{1, n} + f_{2, n}, \ f_{4, n} := f_{1, n} - f_{2, n}.
\end{align*}
Then $\supp \, f_{i, n} \subset [-3, 3]^{d-1} \times [-n-1, n+1]$. 
Now, with similar methods as before, using the fact that the minimal weak upper gradient depends locally on the function, we get that the minimal weak upper gradient $|Df_{i,n}|_w$ of $f_{i, n}$ can be described via
\begin{align*}
     &|Df_{i,n}|_w(z):= \left\{ \begin{array}{ll}
           G_i(z) & \mathrm{if}\ z \in S \times (-n, n), \\
           |\nabla_g f_{i, n}|_g(z) & \mathrm{if}\ z \in ((-3, 3)^{d-1}\setminus S) \times (-n, n), \\
           |Df_{i,1}|_w(z-(n-1)e_d) &\mathrm{if}\ z \in (-3, 3)^{d-1} \times [n, n+1], \\
            |Df_{i,1}|_w(z+(n-1)e_d) &\mathrm{if}\ z \in (-3, 3)^{d-1} \times [-n-1, -n], \\
            0 & \mathrm{otherwise}.
        \end{array}\right.
\end{align*}
Now, 
\begin{align*}
    &\int_M 2|Df_{1,n}|^2_w +2|Df_{2,n}|^2_w-|Df_{3,n}|^2_w-|Df_{4,n}|^2_w \,\di \vol_g \\
    &=  \int_{(-3, 3)^{d-1}\times (-n, n)} 2|Df_{1,n}|^2_w +2|Df_{2,n}|^2_w-|Df_{3,n}|^2_w-|Df_{4,n}|^2_w \,\di \vol_g \\
    &+ \int_{(-3, 3)^{d-1}\times( [-n-1, -n] \cup [n, n+1])} 2|Df_{1,n}|^2_w +2|Df_{2,n}|^2_w-|Df_{3,n}|^2_w-|Df_{4,n}|^2_w \,\di \vol_g \\
    &= n\int_{S\times (-1, 1)} 2G_1^2 +2G_2^2-G_3^2-G_4^2\, \di \vol_g \\
    &+ \int_{(-3, 3)^{d-1}\times((-2, -1)\cup (1, 2))} 2|Df_{1,1}|^2_w +2|Df_{2,1}|^2_w-|Df_{3,1}|^2_w-|Df_{4,1}|^2_w \,\di \vol_g\\
    &= 2n \vol_g(S \times (0,1)) + \int_{(-3, 3)^{d-1}\times((-2, -1)\cup (1, 2))} 2|Df_{1,1}|^2_w +2|Df_{2,1}|^2_w-|Df_{3,1}|^2_w-|Df_{4,1}|^2_w\, \di \vol_g, 
    \end{align*}
where in the last two steps, we first used the fact that $2|\nabla_g f_{1, n}|_g^2 + 2|\nabla_g f_{2, n}|_g^2 -|\nabla_g f_{3, n}|_g^2-|\nabla_g f_{4, n}|_g^2=0$ for all $n \in \N$ and then applied \eqref{non-quadratic_cheegers}.
We have that $|Df_{i,1}|^2_w \in L^\infty$ and hence the last integral in the above equation is finite, so as $\vol_g(S \times (0,1 ))>0$, we get that 
\begin{align}
    \lim_{n \to \infty} 2\Ch(f_{1, n})+ 2\Ch(f_{2, n}) -\Ch(f_{3, n}) -\Ch(f_{4, n}) = \infty. 
\end{align}
In particular, there exists an $n \in \N$ such that $2\Ch(f_{1, n})+ 2\Ch(f_{2, n}) -\Ch(f_{3, n}) -\Ch(f_{4, n}) \neq 0$.
\end{proof}

This proves Theorem \ref{hilbertianity_counterexample}, taking into account that for $p' \in [1, d-1]$, one has $W^{1,p'}_{\rm loc} \subset W^{1,d-1}_{\rm loc}$. 

\smallskip
\begin{remark}
    For $d \geq 4$, the above example shows that we can find a Riemannian metric $g$ on $\R^d$ of the Geroch-Traschen class that induces a non-infinitesimally Hilbertian metric measure space.
\end{remark}
\smallskip

\begin{remark}
    For a manifold with a $W^{1,p}_{\rm loc}$-Riemannian metric such that $p>d$, the Sobolev embedding theorem together with the results from \cite{mondino2025equivalence} immediately give us that $(M, \sfd_g, \vol_g)$ is infinitesimally Hilbertian. Theorem \ref{hilbertianity_counterexample} shows that $g, g^{-1} \in L^\infty_{\rm loc} \cap  W^{1,p}_{\rm loc}$ for $p \leq d-1$ is not sufficient for infinitesimal Hilbertianity. This works towards answering Question \ref{q3} leaving only the interval of $p \in (d-1, d]$, where it is unknown whether $g, g^{-1} \in L^\infty_{\rm loc} \cap  W^{1,p}_{\rm loc}$ is sufficient to get infinitesimal Hilbertianity. 
\end{remark}
\smallskip

\begin{remark}
    One motivation behind the definition of infinitesimal Hilbertianity was to distinguish Finslerian structures from Riemannian ones, in particular in the context of synthetic lower Ricci curvature bounds. The distance in this example arises from a Riemannian metric of low regularity; however the corresponding metric measure space does not satisfy infinitesimal Hilbertianity.
    The directional dependence of the metric speed of curves established in \eqref{asymptotic_lower_bound_metric_ex2} reveals the asymptotic shape of metric balls in the manifold and it is readily seen that the space $(M, \sfd_g, \vol_g)$ is not locally Minkowskian (see \cite{ohta2012non} and \cite{magnabosco2025canonical} for a discussion of infinitesimal Hilbertianity for locally Minkowskian spaces).  
    Moreover, the construction immediately implies that this space does not satisfy neither distributional lower Ricci curvature bounds nor the $\cd(K, \infty)$-condition for any $K \in \R$.
\end{remark}
\section*{Statements and Declarations}
\textbf{Conflict of interest statement.} The author has no relevant financial or non-financial
interests to disclose. The author has no conflict of interest to declare that are relevant to the content of this article.

\textbf{Data Availability.} Data sharing is not applicable to this article as no new data were created or analysed in this study.

\phantomsection
\addcontentsline{toc}{section}{References}
\bibliography{bibliographie}
\bibliographystyle{abbrv}
\end{document}